\documentclass[]{amsart}
\usepackage{amssymb}
\usepackage{graphicx}
\usepackage{xcolor} 
\usepackage{fullpage} 
\usepackage{amsmath}
\usepackage{amsthm}
\usepackage{verbatim}
\usepackage{enumitem}

\newtheorem{theorem}{Theorem}
\newtheorem{lemma}[theorem]{Lemma}
\newtheorem{prop}[theorem]{Proposition}
\newtheorem{corollary}[theorem]{Corollary}
\newtheorem*{theorem*}{Theorem}
\newtheorem*{lemma*}{Lemma}

\theoremstyle{definition}
\newtheorem{definition}{Definition}
\newtheorem{remark}{Remark}

\newtheorem*{remark*}{Remark}
\newtheorem*{example*}{Example}
\newtheorem*{er*}{Examples and Remarks}

\newcommand{\nrm}[1]{\Vert#1\Vert}

\newcommand{\nnrm}[1]{{\vert\kern-0.25ex\vert\kern-0.25ex\vert #1 \vert\kern-0.25ex\vert\kern-0.25ex\vert}}

\newcommand{\supp}{{\mathrm{supp}}\,}

\newcommand{\lap}{\Delta}

\newcommand{\ud}{\mathrm{d}}
\newcommand{\rd}{\partial}

\newcommand{\0}{\emptyset}

\newcommand{\eps}{\epsilon}

\begin{document}

\title[vortex-stretching]
{A remark on the zeroth law and instantaneous vortex stretching 
on the incompressible 3D Euler equations} 
 
\author{In-Jee Jeong}
\address{Department of Mathematics, Korea Institute for Advanced Study}
\email{ijeong@kias.re.kr}

\author{Tsuyoshi Yoneda} 
\address{Graduate School of Mathematical Sciences, University of Tokyo, Komaba 3-8-1 Meguro, 
Tokyo 153-8914, Japan } 
\email{yoneda@ms.u-tokyo.ac.jp} 


\date{\today} 


\begin{abstract} 
By DNS of Navier-Stokes turbulence, Goto-Saito-Kawahara (2017) showed that turbulence consists of a self-similar hierarchy of anti-parallel pairs of vortex tubes, in particular, stretching in larger-scale strain fields creates smaller-scale vortices. Inspired by their numerical result, we examine the Goto-Saito-Kawahara type of vortex-tubes behavior using the 3D incompressible Euler equations, and show that such behavior induces energy cascade in the absence of nonlinear scale-interaction. From this energy cascade, we prove a modified version of the zeroth-law. In the Appendix, we derive Kolmogorov's $-5/3$-law from the GSK point of view. 
\end{abstract}

\maketitle

\section{Introduction}

An important problem in the study of turbulence is to figure out a concrete 
picture of the self-similar hierarchy and clarify its mechanism, that is,
the energy cascade dynamics.
Goto-Saito-Kawahara \cite{GSK} (GSK) qualitatively 
examined that sustained turbulence consists of a hierarchy of antiparallel
pairs of vortex tubes. 
The main problem was in isosurface visualizations of vortices.
To overcome this difficulty, they used the low-pressure method proposed by Miura-Kida \cite{MK1,MK2}.
This threshold-free method identifies the axis of a vortex tube using the fact
that the pressure field take a local minimum at the center of a vortex tube 
on the plane perpendicular to it.
Using this method, they showed that vortex tubes at any inertial length 
scale tend to form antiparallel pairs with any type 
of large-scale forces, 
and 
 no qualitative difference in
these structures.
\begin{itemize}
\item
The first conclusion of GSK is that
 turbulence
(in the inertial length scales) is composed of hierarchy of vortex tubes with different sizes.
\end{itemize}
GSK also observed that the hierarchy is mainly created by vortex stretching.
Smaller
vortices are created around larger vortices, and 
their finding is that not a single vortex tube but antiparallel pairs of them effectively stretch
and create smaller-scale vortices at each hierarchical level.
They also emphasize that this antiparallel tendency of vortices is
independent of the external force. 
\begin{itemize}
\item
The second conclusion of GSK is that,
at each hierarchy level, vortex tubes tend to form antiparallel pairs and they effectively
stretch and create smaller-scale vortex tubes.
Also stretched vortex tubes tend to align in the direction perpendicular to larger-scale vortex tubes.
\end{itemize}
(The above conclusion is the key in this paper, thus, the same sentences are repeated a few times.)
It is also important to see whether or not such stretching process is local in scale, because  the scale-locality of the energy cascade  
provides 
the  universality of small-scale statistics.
\begin{itemize}
\item
The third conclusion of GSK is that vortices at each hierarchical level are 
most likely to be stretched in strain fields around $2$-$8$ times larger vortices.
\end{itemize}
In this paper we examine GSK-type of vorticity to see the energy cascade mechanism in the fully-nonlinear level. 
In order to clarify (for the readers) such fully-nonlinear framework, 
first, we explain the classical view of the scale-by-scale energy 
cascade in the inertial length scales.
More precisely, we impose linearizing assumptions: scale-locality and space-locality,
and show that
 vortex stretching leads to forward energy cascade.
In order to do so, first let us decompose the velocity field  $u(t,x)$ into mean and fluctuating parts: the mean part $\bar u_K$ whose length-scale is larger than $1/K$ is defined by ($d$ is dimension)
\begin{equation}\label{Gauss}
\text{(coarse-graining)}\quad\bar u_K(t,x):=\int_{\mathbb{R}^d}G_K(r)u(t,x+r)dr,
\end{equation}
where $G$ is the Gaussian function $G(r)=C\exp(-|r|^2)$ with $G_K(r)=K^{d}G(Kr)$. We plug $\bar u_K$ into the usual Navier-Stokes equations:
\begin{equation}\label{NS}
\partial_t u+(u\cdot \nabla) u=-\nabla  p+\nu\Delta u,\quad\text{div}\, u=0,
\end{equation}
and then we obtain
\begin{equation}\label{averaged-NS}
\partial_t\bar u_K+(\bar u_K\cdot \nabla)\bar u_K+\nabla \cdot \tau_K=-\nabla \bar p_K+\nu\Delta\bar u_K
,\quad\text{div}\, \bar u_K=0.
\end{equation}
Here, $\tau_K=\overline{(u\otimes u)}_K-\bar u_K\otimes\bar u_K$ is called stress tensor, which is determined by the fluctuating part whose scale is less than $1/K$.
In the turbulence study field, looking into energy transfer in the inertial range is the most important, and in this case we do not need to see the viscosity effect. Thus we may take $\nu$ sufficiently small enough. We multiply by $\bar u_K$ and integrate on both sides, to obtain 
\begin{equation*}
\begin{split}
\frac{1}{dt}\frac{1}{2}\int_{\mathbb{R}^d}|\bar u_K(t,x)|^2 dx+\nu \int_{\mathbb{R}^d} |\nabla \bar u_K(t,x)|^2dx & = \int_{\mathbb{R}^d} \bar S_K(t,x): \tau_K(t,x)dx \\
& =: -\int_{\mathbb{R}^d}\Pi_K(t,x)dx, 
\end{split}
\end{equation*}
where $\Pi_K$ is called the energy-flux and  $\bar S_K$ is called deformation tensor:
\begin{equation*}
\bar S_K=\frac{1}{2}\left[(\nabla \bar u_K)+(\nabla \bar u_K)^T \right], 
\end{equation*}
where $(\cdot)^T$ represents the corresponding transposed matrix, and
$A:B = \sum_{i,j=1}^da_{ij}b_{ij}$ for $A=(a_{ij})_{i,j=1}^d$ and $B=(b_{ij})_{i,j=1}^d$. If this energy-flux $\Pi_K$ is a positive-valued function, then the large-scale energy
 $\int |\bar u_K(t,x)|^2dx$ decreases, thus this is so-called ``energy cascade" or ``forward energy cascade". To the contrary, if $\Pi_K$ is negative, then $\int |\bar u_K(t,x)|^2dx$ increases, thus this expresses ``inverse energy cascade".
To figure out whether energy cascades forward or inverse,
we need to designate the typical turbulence picture, and 
apply it to the energy-flux $\Pi_K$.
Thus it is widely believed that, we cannot obtain 
any reasonable finite-dimensional approximation of the
fluid equation without such designation of the turbulence picture (it is the so-called ``closure problem").
In the classical manner, at this stage, we linearize the energy-flux $\Pi_K$ using 
scale-locality and space-locality. Let us explain more precisely.
First we decompose $\tau_K$.
To do so, let us define the exact $1/n$-scale flow $u^{[n]}$ as follows:
\begin{equation*}
u^{[n]}=\bar u_{n}-\bar u_{n-1}\quad (n\geq 2),\quad
 u^{[1]}=\bar u_1.
\end{equation*}
Then formally, we obtain 
\begin{equation*}
\tau_K=\sum_{n=1}^\infty\sum_{n'=1}^\infty\overline{(u^{[n]}\otimes u^{[n']})_K}-\overline{(u^{[n]})_K}\otimes \overline{(u^{[n']})_K}.
\end{equation*}
Here we apply the assumption of ``scale-locality"; that is, the
interaction between $u^{[n]}$ and $u^{[n']}$ for $(n\not= n')$ is small enough, and the interaction between $\bar u^{K}$ and $u^{[k]}$ is dominant. Then we can get the following approximation (see \cite{E06-0}):
\begin{equation}\label{stress-tensor-0}
\tau_K\approx \tau^{[k,k]}:=\overline{(u^{[k]}\otimes u^{[k]})_K}-\overline{(u^{[k]})_K}\otimes \overline{(u^{[k]})_K}\quad(k>K).
\end{equation}
By the third conclusion in GSK, we can regard $k\sim bK$ for some $b\in [2,8]$.

\begin{remark}
If we use the Littlewood-Paley decomposition as the coarse-graining, then  the second term in \eqref{stress-tensor-0}
essentially disappears.
In this case, we can regard
\begin{equation}\label{LP-version}
\tau^{[k,k]}=\overline{(u^{[k]}\otimes u^{[k]})_K}.
\end{equation}
We use this kind of approximation in the Appendix.
More precisely, we derive Kolmogorov's $-5/3$-law, from GSK-type 
of self-similar hierarchy.
\end{remark}
The following is just a consequence of direct calculation.
\begin{lemma}[{{Eyink \cite[(2.11)]{E06-0}}}]
Let $\delta u:=\delta u(r;x)=u(x+r)-u(x)$.
Then we have 
\begin{equation}\label{toushiki}
\tau^{[k,k]}=\overline{(\delta u^{[k]}\otimes\delta u^{[k]})_K}-\overline{(\delta u^{[k]})_K}\otimes \overline{(\delta u^{[k]})_K}. 
\end{equation}
\end{lemma}
Here we also apply the another assumption of ``space-locality"; that is, for each $\delta u^{[n]}$, we approximate it to the $m$-th order Taylor expansion:
\begin{equation*}
\delta u^{(k,m)}(r;x):=
\sum_{p=1}^m\frac{1}{p!}\left(r_1\frac{\partial}{\partial x_1}+\cdots+r_d\frac{\partial}{\partial x_d}\right)^pu^{[k]}(x)
\end{equation*}
(see Eyink \cite[Section 3.2.2]{E06-0}).
Then we have
\begin{equation*}
\tau^{[k,k]}\approx \tau^{(k,m)}
:=\overline{(\delta u^{(k,m)}\otimes\delta u^{(k,m)})_K}-\overline{(\delta u^{(k,m)})_K}\otimes \overline{(\delta u^{(k,m)})_K},
\end{equation*}
and for $m = 1$, we obtain the following first-order approximation:
\begin{equation}\label{first-order}
\tau_{ij}^{(k,1)}
=
C_K\sum_{h=1}^d\frac{\partial u_i^{[k]}}{\partial x_h}\frac{\partial u_j^{[k]}}{\partial x_h}
\end{equation}
for some positive constant $C_K>0$ depending only on $K$.
 To calculate the energy-flux $\Pi_K$ using the above approximation \eqref{first-order}, 
 we use GSK-type of turbulence picture.
In particular, to set up the small-scale and large scale flows, we follow the second conclusion in GSK: strongly stretched vortex tubes tend to align in the direction perpendicular to larger-scale vortex tubes.
Let $r=\sqrt{x_1^2+x_2^2}$ and examine
small-scale vortex tube $u^{[k]}$ as follows:
\begin{equation}\label{classical-small-scale}
\bar u^{[k]}(x)=
\begin{cases}
\begin{pmatrix}
x_2\\-x_1\\0
\end{pmatrix}
\quad (r\le k^{-1}), \\
\dfrac{1}{(rk)^2} 
\begin{pmatrix}
x_2\\-x_1\\0 
\end{pmatrix}
\quad (r>k^{-1}).
\end{cases}
\end{equation}
Also we examine large scale strain velocity $\bar u_K$ as 
\begin{equation}\label{classical-large-scale}
\bar u_K(x)=
\begin{pmatrix}
0\\
- x_2\\
x_3
\end{pmatrix}.
\end{equation}
In this setting  we do not put any $K$-factor in the definition (since it is not important anymore).
This $\bar u_K$ represents anti-parallel vortex tube in the $x_1$-direction.
Then we have
\begin{equation}\label{flux-matrix}
\tau^{(k,1)}(x)
=
C_K
\begin{pmatrix}
1 &0&0\\
0 &1&0\\
0& 0&0
\end{pmatrix}
\times
\begin{cases}
1\quad & (r\leq k^{-1})\\
1/(rk)^4\quad & (r> k^{-1}).
\end{cases}
\end{equation}
This \eqref{flux-matrix} is an entirely positive function and since
\begin{equation*}
\bar S_K(x)=
\begin{pmatrix}
0& 0&0\\
0&-1&0\\
0&0& 1
\end{pmatrix},
\end{equation*}
we see 
\begin{equation}\label{forward energy cascade}
\Pi_K(x)=-\tau^{(k,1)}(x):\bar S_K(x)>0.
\end{equation}
This exhibits forward
energy cascade.
However, in the above classical observation, we are already imposing linearizing 
approximations: scale-locality and space-locality. 
Thus our basic question is that what is happening in the small-scale large-scale interaction in the fully-nonlinear level.
To answer this question, we 
also focus on the second conclusion in GSK: stretched small-scale vortex tubes tend to align in the direction perpendicular to larger-scale vortex tubes, and employ the incompressible  3D Euler equations.
Also, from this consideration, we try to derive a ``modified zeroth-law" using stretching of certain antiparallel pairs of vorticity, which could be a
useful approach to look into the actual zeroth-law in the future.
The zeroth-law states that, in the limit of the vanishing viscosity, the kinetic energy  
dissipation becomes nonzero; that is, for the solution $u$ (depending on $\nu$)  to the Navier-Stokes equations \eqref{NS}, 
\begin{equation*}
\lim_{\nu\to 0}\nu\langle|\nabla u|^2\rangle=\epsilon>0,
\end{equation*}
where $\langle\cdot\rangle$ is some averaging (in this paper we regard it as both space and time averages).
Laboratory experiments and numerical simulations of turbulence both confirm 
the above zeroth-law (see Eyink \cite{E3}). Thus it is widely believed that the zeroth-law is the cornerstone 
in the turbulence study field (for a recent development related to  the zeroth-law, see Drivas \cite{Dr}).

\section{The 3D vorticity equations and main theorems.}

The 3D vorticity equations (equivalently, the 3D Euler equations) take the following form: 
\begin{equation*}
\begin{split}
\partial_t\omega + (u\cdot\nabla) \omega = (\omega\cdot \nabla) u,\quad x \in \mathbb{T}^3 := (\mathbb{R}/2\mathbb{Z})^3 
\end{split}
\end{equation*} where the velocity $u$ is determined by the (periodic) 3D Biot-Savart law: \begin{equation*}
\begin{split}
u(t,x) = \int_{\mathbb{T}^3} K_3 \large( x - y \large) \omega(t,y) \, dy, 
\end{split}
\end{equation*} with \begin{equation*}
\begin{split}
K_3(x)v  = \frac{1}{4\pi} \frac{x \times v}{|x|^3} \quad (\text{with reflections}).
\end{split}
\end{equation*} The associated Lagrangian flow is then given by \begin{equation*}
\begin{split}
\partial_t \Phi(t,x)=u(t,\Phi(t,x))\quad\text{with}\quad \Phi(0,x)=x\in\mathbb{T}^3.
\end{split}
\end{equation*} 
We shall examine smooth initial vorticity of the form $\omega_{0,n} = \omega_{0,n}^L + \omega_{0,n}^S$ for $\omega^L_{0,n}\in C^\infty_{c}$ and $\omega^S_{0,n}\in C^\infty_{c}$,
and we  restrict them to the following symmetry (with a slight abuse of notation):
\begin{equation*}
\omega^L_{0,n}=(0,0,\omega_{0,n}^L(x_1,x_2))\quad
\text{and}\quad 
\omega_{0,n}^S=(\omega_{0,n,1}^S(x_1,x_2), \omega_{0,n,2}^S(x_1,x_2),0)
\end{equation*}
(c.f. \eqref{classical-small-scale} and \eqref{classical-large-scale}).
The corresponding solution also keeps this symmetry.
This setting has been inspired by the second conclusion in GSK.
Let $\omega_{n}(t)$ be the corresponding solution to $\omega_{0,n}$. By the Biot-Savart law,
\begin{eqnarray*}
\nonumber
u_{n}(t,x)
= 
\int_{\mathbb{T}^3} K_3\big(x - y \big) \omega_{n}(t,y) \, dy,\
\partial_t \Phi_{n}(t,x)=u_{n}(t,\Phi_{n}(t,x))
\end{eqnarray*} 
and then
\begin{equation*} 
\omega_{n}(t,\Phi_{n}(t,x)) 
= 
 D\Phi_{n}(t,x) \omega_{0,n}(x)
=D\Phi_{n}(t,x)(\omega^L_{0,n}(x)+\omega_{0,n}^S(x)),
\end{equation*} 
where
\begin{equation*}
D\Phi_n:=
\begin{pmatrix}
\partial_1\Phi_{n,1}&\partial_2\Phi_{n,1}&\partial_3\Phi_{n,1}\\
\partial_1\Phi_{n,2}&\partial_2\Phi_{n,2}&\partial_3\Phi_{n,2}\\
\partial_1\Phi_{n,3}&\partial_2\Phi_{n,3}&\partial_3\Phi_{n,3}
\end{pmatrix}.
\end{equation*}
Moreover, since the third variable of each components is absent, we can decouple the equations into two parts: 2D-part and third component part (see \cite[Section 1.3]{EM} for example).
More precisely,
we take $\omega^L_{0,n}(x_1,x_2,x_3) = (0, 0, \tilde{\omega}^L_{0,n}(x_1,x_2))$, where $\tilde{\omega}^L_{0,n}$ is from Lemma \ref{lem:LLD}. The initial data $\tilde{\omega}^L_{0,n}$ gives rise to a 2D flow map $\eta_n$. With some abuse of notation, denoting the 3D flow map associated with the solution for $\omega^L_{0,n}$ again by $\eta_n$, we have the representation 
\begin{equation*}
D\eta_n:=
\begin{pmatrix}
\partial_1\eta_{n,1}&\partial_2\eta_{n,1}&0\\
\partial_1\eta_{n,2}&\partial_2\eta_{n,2}&0\\
0&0&1
\end{pmatrix},
\quad
\quad
D\eta_n^{-1}=
\begin{pmatrix}
\partial_2\eta_{n,2}&-\partial_2\eta_{n,1}&0\\
-\partial_1\eta_{n,2}&\partial_1\eta_{n,1}&0\\
0&0&1
\end{pmatrix}.
\end{equation*}
 Note that we always have $\nrm{D\eta_n(t)}_{\infty} = \nrm{D\eta_n^{-1}(t)}_{L^\infty}$. We shall also identify $D\eta_n$ and $D\eta_n^{-1}$ with $2\times 2$ matrices whenever it is convenient to do so.
Eventually we have the following explicit formula: 
\begin{equation}\label{scale-separation-expression}
	\begin{split}
	\omega_n^L(t,\eta_n(t,x)) = \omega_{0,n}^L(x)\quad\text{and}\quad \omega_n^S(t,\eta_n(t,x)) = D\eta_n(t,x)\omega_{0,n}^S(x).
	\end{split}
	\end{equation}
Again, by the Biot-Savart law, we can also recover the large-scale velocity:
\begin{equation*}
\nonumber
 u_{n}^L(t,x)
= 
\int_{\mathbb{T}^3} K_2\big(x - y \big) \omega_{n}^L(t,y) \, dy
\end{equation*}  
and also $u_{n}^S(t,x)
=
 u_{0,n}^S(t,\eta_n(t,x))$ with $\nabla\times u_{0,n}^S=\omega_{0,n}^S$.
\begin{remark}
 On the other hand, in Bourgain-Li's illposed construction, 
the authors are carefully controlling the scale-interaction (see \cite[Sec. 4]{BL}, in particular, Lemma 4.1 combining $L^4$-energy (4.16)). This must be the crucial difference on the scale-interaction between 2D and 3D.
\end{remark}
Since 
 $\omega^S_{n}(t)\perp e_3$ and $\omega^L_{n}(t)\parallel e_3$,  
 we have $$\|\omega(t)\|_{L^2}^2=\|\omega^L_{n}(t)\|_{L^2}^2+\|\omega^S_{n}(t)\|_{L^2}^2.$$
In order to see the energy cascade mechanism, we reasonably define expected-value of spectrum on the small-scale energy $\mathbb E[u^S_n]$ as follows:
\begin{equation*}
\mathbb E[u^S_n(t)]:=\frac{\int_{\mathbb{T}^3} |\xi|^2|\hat u^S_n(\xi)|^2d\xi}{\|u^S_n(t)\|_{L^2}^2}\approx \frac{\|\omega^S_n(t)\|_{L^2}^2}{\|u^S_n(t)\|_{L^2}^2}.
\end{equation*} 
Roughly saying, in this case, the spectrum ($\sim 1/(scale)$) concentrates at $\mathbb E[u^S_n(t)]$, and thus we can use this value as the scale.
We now give the main theorem.
\begin{theorem}[Instantaneous energy cascade]\label{thm:main}\  
Let $\omega^L_{0,n}$ be either \eqref{BC-initial-data} or \eqref{eq:bubbles-Linfty}, which is bounded uniformly in $L^\infty$.
Then there is a sequence of  points and radii $\{x^*_n\}_n, \{r_n\}_n\subset \mathbb{T}^2$ such that the following statement holds:
For  $\{\rho_n\}_n\subset C^\infty_{c}(\mathbb{T}^2)$ 
with $\|\nabla\rho_n\|_{L^2} = 1$,  let $\omega^S_{0,n} \in C^\infty(\mathbb{T}^3)$ be a sequence of smooth small-scale initial vorticity such that
\begin{equation}\label{sequence-of-small-scale} 
u_{0,n}^S=(0,0,   
\rho_n(x_h-x_n^\ast)),\quad 
\omega^S_{0,n}=\nabla\times u_{0,n}^S
\quad\text{and}\quad \text{supp}\, \rho_n\subset B(x^*_n,r_n),
\end{equation} 
where  $x_h=(x_1,x_2)$.
Then there exist $\delta > 0$  and a sequence of positive time moments $t_n\to t_\infty\in[0,\delta]$   such that the corresponding solution satisfies
$\|u^S_n(t)\|_{L^2}=\|u^S_{0,n}\|_{L^2} \lesssim  1$,
\begin{equation*}
\|\omega_{0,n}^S\|_{L^2} \approx 1
\quad\text{and}\quad
\frac{\mathbb E[u^S_n(t_n)]}{\mathbb E[u^S_n(0)]}\to\infty
\quad\text{as}\quad
n\to\infty.
\end{equation*}

\end{theorem}
In the proof, one can see that this instantaneous energy cascade is induced by vortex-stretching. 
\begin{remark}
	Note that \begin{equation*}
	\begin{split}
	u^L_n(t,x) \approx a_n(t)\begin{pmatrix}
	x_1 \\
	-x_2
	\end{pmatrix}
	\end{split}
	\end{equation*} near some neighborhood of the origin, with $a_n(t) \rightarrow \infty $ as $n \rightarrow \infty$. Then if $u_n^S$ were initially supported in that neighborhood, we have $u_n^S(t) \approx v_n^S(t)$ where $v_n^S$ is defined by \begin{equation*}
	\begin{split}
	\rd_t v_n^S + a_n(t)\begin{pmatrix}
	x_1 \\
	-x_2
	\end{pmatrix} \cdot\nabla v_n^S = 0 
	\end{split}
	\end{equation*} with $v_{0,n}^S = u_{0,n}^S$, which is explicitly solvable with \begin{equation*}
	\begin{split}
	v^S_n(t,x_1,x_2) = u_{0,n}^S(e^{-\int_0^t a_n(s)ds}x_1,e^{\int_0^t a_n(s)ds}x_2 ).
	\end{split}
	\end{equation*} From this expression, energy cascade is clear. Moreover, we can rephrase the above theorem by using the coarse graining \eqref{Gauss}, namely (here we write down in the whole space $\mathbb{R}^2$ case),
\begin{equation}
\|\overline{(v_n^S)}_K(t_n)\|^2_{L^2}=
\int_{\mathbb{R}^2}|\hat u_{0,n}^S(t_n,e^{\int_0^t a_n(s)ds}\xi_1,e^{-\int_0^{t_n} a_n(s)ds}\xi_2 )
\hat G_K(\xi)|^2d\xi
\to 0
\end{equation}
as $ n \rightarrow \infty$ for any $K$ (c.f. (17) and (18) in \cite{DE}, see also Remark \ref{Theodore-Eyink} in this paper).
\end{remark}

\begin{remark}
At a glance, this mechanism looks similar to \cite[Proposition 1.3]{EM} but totally different.
In their argument, $L^2$-norm of vorticity case is excluded. They just directly apply Lagrangian flow created by Bahouri-Chemin vorticity (see \eqref{eq:BC}).
On the other hand, in our argument, we need to carefully estimate the Lagrangian deformation in some time interval.
Let us further clarify: we can also show that there is a sequence of small-scale initial vorticity with $\|\omega^S_{0,n}\|_{L^2}\to 0$ ($n\to\infty$) such that the corresponding solutions satisfy
$\|\omega^S_n(t_n)\|_{L^2}\to\infty$ ($n\to\infty$). However their argument cannot capture this mechanism.  
\end{remark}

	\begin{remark}
		We present a simple computation which illustrates that, at least in the setup of $2+\frac{1}{2}$-dimensional flow, anomalous dissipation of energy is not caused by the creation of vortex-stretching by $C^{1/3}$-roughness. To this end, recall that  \begin{equation}\label{eq:2andhalfEuler}
		\left\{ \begin{aligned}
		&\rd_t u_h + u_h\cdot\nabla u_h + \nabla p = 0, \\
		&\rd_t u_3 + u_h\cdot\nabla u_3 = 0,
		\end{aligned}\right.
		\end{equation} where $u = (u_h,u_3)$ is a function of $(x_1,x_2)$ only.  We take $\omega_h = \nabla\times u_h$, and consider the initial data \begin{equation*}
		\begin{split}
		\omega_{h,0}= r^{-\frac{2}{3}}\sin(2\theta), 
		\end{split}
		\end{equation*} where $(r,\theta)$ is the usual polar coordinates. Then we compute that the corresponding velocity is \begin{equation*}
		\begin{split}
		u_{h,0} = \nabla^\perp\lap^{-1}(\omega_{h,0}) =c  x_h^\perp \left( r^{-2/3} - \frac{2}{3} r^{-2/3-2} x_1x_2 \right). 
		\end{split}
		\end{equation*} Note that $u_{h,0}$ belongs to exactly $C^{1/3}(\mathbb{R}^2)$ and not better. We may further set $u_{3,0}$ to behave like $|x_h|^{1/3}$ near the origin. However that the energy is conserved (at least at the initial time). Indeed, \begin{equation*}
		\begin{split}
		\frac{1}{2}\left.\frac{d}{dt}\right\rvert_{t = 0}  \int  u_h^2 + u_3^2 \, \ud x_1\ud x_2\, & = - \int u_{h,0}\cdot \left(\nabla u_{h,0} \cdot u_{h,0} + \nabla u_{3,0}\cdot u_{3,0} \right) \, \ud x_1\ud x_2\, \\
		& = \lim_{\epsilon\rightarrow 0} \int_{ \rd B(0,\epsilon)} \frac{1}{2} u_{h,0} \cdot N \left( |u_{h,0}|^2 + |u_{3,0}|^2 \right) 
		\, \ud \sigma \\
		& = \lim_{\epsilon\rightarrow 0}  O(\epsilon^{1+ 3(1-2/3)}) = 0,
		\end{split}
		\end{equation*} where $N$ is the outwards unit normal vector on $\rd B(0,\epsilon)$ and $\sigma$ is the Lebesgue measure on $\rd B(0,\epsilon)$. We refer to recent works of Luo and Shvydkoy \cite{LS1, LS2, Sh} which systematically studies the radially homogeneous solutions to 2D and 3D Euler equations and conclude absence of anomalous dissipation in that class of solutions. \end{remark}

{ 
	We now give the second theorem. It states that the instantaneous energy cascade we have obtained in the above result is strong enough to create dissipation of order $\nu^{1-a_0}$ for some $1 > a_0 > 0$ in a viscosity independent time interval for the 3D Navier Stokes equations in the limit $\nu \rightarrow 0^+$. For this we shall take the Bahouri-Chemin type data \eqref{BC-initial-data} as well as a concrete choice of $\rho_n$ in \eqref{sequence-of-small-scale}; see below. 
}

\begin{theorem}[Modified zeroth-law] \label{thm:main2}
Let $\{\omega^\nu_n(t)\}_{\nu,n}$ be a 
vorticity sequence of the 3D Navier-Stokes equations with
$\omega^\nu_n(0)=\omega_{0,n}^L+\omega_{0,n}^S$,
where { $\omega_{0,n}^L$ is the smoothed Bahouri-Chemin data as in \eqref{BC-initial-data}
and $\omega_{0,n}^S$ is given by $\nabla\times u_{0,n}^S = \nabla\times (0,0,\rho_n(x_h))$, where $\rho_n \ge 0$ is a smooth function satisfying $\rho_n(x_h) = 2^{n/2} x_2$ in $B(0,2^{-n/2})$, $\rho_n = 0$ outside of $B(0,2^{2-n/2})$, and $\nrm{\nabla \rho_n}_{L^\infty} \le 2^{n/2}$.} The sequence of initial velocity is convergent in $L^2(\mathbb{T}^3)$. Then, there exist a small constant $0<a_0 <1$ and a sequence of viscosity constants $\nu_n > 0$, converging to zero, 
such that
\begin{equation}\label{zeroth-law-estimate}
\liminf_{n\to \infty}\, \nu_n^{a_0} \frac{1}{\delta}\int_0^\delta\int_{\mathbb T^3}|\omega^{\nu_n}_n(t,x)|^2dxdt>1,
\end{equation}
where $\delta$ is determined in Proposition \ref{BC}.
\end{theorem}

The open problem is to find a sequence of vorticity which achieves $a_0=1$.
{
\begin{remark}
 The estimate \eqref{zeroth-law-estimate} is dimensionally not correct, in the sense that
 left and right hand sides
do not have the same physical dimensions.
This is because  we are using the  initial data which have different scalings in large-scale  and small-scale. 
\end{remark}
}

\begin{remark}
	Note that in the above statement, the sequence of initial velocity is uniformly bounded in $H^1$ and hence convergent in any norms between $L^2$ and $H^1$. It is not difficult to guarantee that, by rescaling $\omega_{0,n}^S$ appropriately in $n$, that the same statement holds (possibly with a smaller $a_0 > 0$) with initial velocity uniformly bounded and convergent in $W^{1,p}$ with some $p > 2$. 
\end{remark}
{
\begin{remark}\label{Theodore-Eyink}
We compare \eqref{zeroth-law-estimate} with a recent result on an Onsager singularity theorem for Leray solutions of incompressible Navier-Stokes equations \cite{DE}.
The authors of \cite{DE} showed that if a sequence of Leray solutions $\{u^\nu\}_\nu$ are uniformly bounded in
$L^3([0,\delta];B^\sigma_{3,\infty}(\mathbb{T}^3))$ for some $\sigma\in (0,1)$,
then the corresponding solutions satisfy
\begin{equation}\label{dissipation estimate}
\int_0^\delta\int_{\mathbb{T}^3}\epsilon[u^\nu]dxdt\lesssim \nu^{\frac{3\sigma-1}{\sigma+1}},
\end{equation}
where $\epsilon[\cdot]$ is expressing energy dissipation, that is (in the absence of the external force),
\begin{equation*}
\int_0^\delta\int_{\mathbb{T}^3}\epsilon[u^\nu]dxdt:=
\frac{1}{2}\int_{\mathbb{T}^3}|u^\nu(0,\cdot)|^2dx-\frac{1}{2}\int_{\mathbb{T}^3}|u^\nu(\delta,\cdot)|^2dx.
\end{equation*}
(Note that the function space $L^3([0,\delta];B^\sigma_{3,\infty}(\mathbb{T}^3))$ is physically natural; see Remark 1 in \cite{DE}.) The estimate \eqref{dissipation estimate} gives an upper bound on the value of the constant $a_0$ from \eqref{zeroth-law-estimate}: for $\sigma> (2-a_0)/(2+a_0)$, the sequence of solutions $\{u_n^{\nu_n}\}_n$ in Theorem \ref{thm:main2} (the corresponding vorticities are $\{\omega_n^{\nu_n}\}_n$) does not belong to $L^3([0,\delta];B^\sigma_{3,\infty}(\mathbb{T}^3))$  uniformly in $n$.
The proof is the following: 
assume to the contrary that the sequence of solutions $\{u_n^{\nu_n}\}_n$ in Theorem \ref{thm:main2} belongs to $L^3([0,\delta]; B^\sigma_{3,\infty}(\mathbb{T}^3))$ uniformly in $n$. 
 By \eqref{zeroth-law-estimate}, we see
\begin{equation*}
\int_0^\delta\int_{\mathbb{T}^3} \epsilon[u^{\nu_n}_n] dxdt = \nu_n\int_0^\delta\int_{\mathbb{T}^3}|\nabla u^{\nu_n}_n(t,x)|^2dxdt\gtrsim\nu_n^{1-a_0}.
\end{equation*}
Thus, if $\sigma$ satisfies  $1-a_0>\frac{3\sigma-1}{\sigma+1}$, that is,
 $\sigma> (2-a_0)/(2+a_0)$, then it contradicts \eqref{dissipation estimate} for sufficiently large $n$.
On the other hand,  the sequence of solutions $\{u^{\nu_n}_n\}_n$  belongs uniformly in $L^3_t B^\sigma_{3,\infty}$ with some 
$\sigma$. To see this, one can directly estimate the equation \begin{equation*} 
\begin{split}
&\rd_t \omega_n^{S,\nu} + u_n^{L,\nu} \cdot \nabla \omega_n^{S,\nu} = \nabla u_n^{L,\nu} \omega_n^{S,\nu} + \nu \lap \omega_n^{S,\nu}
\end{split}
\end{equation*} in $L^p$: $\nrm{\omega_n^{S,\nu}(t)}_{L^p} \lesssim \nrm{\omega_{0,n}^S}_{L^p} \exp(\int_0^t \nrm{\nabla u_n^{L,\nu}(s)}_{L^\infty} ds)$, with an implicit constant independent of $\nu \ge 0$. From our choice of initial data and $\nrm{\nabla u_n^{L,\nu}}_{L^\infty} \lesssim n$, it follows that the corresponding solution $\omega_n^{S,\nu}$ belongs  to $L^\infty([0,t];L^{p(t)}(\mathbb{T}^2))$ with $p(t) = 2-ct$ for some constant $c>0$. 
Then at least for $t > 0$ sufficiently small, the velocity is uniformly in $L^\infty_t W^{1,p(t)} \subset L^3_t B^\sigma_{3,\infty}$ with $2-3/p(t) = \sigma$.  This gives a restriction that $a_0 \le 2/3$.
\end{remark}
}

\begin{remark}
Let us also explain briefly a recent result on the mathematical turbulence.
Buckmaster-De Lellis-Sz\'ekelyhidi-Vicol \cite{BLSV} showed the following:
For any positive smooth function $e: [0,T]\to \mathbb{R}$, and for any
$0<\beta<1/3$, there is a weak Euler solution $v\in C^\beta([0,T]\times\mathbb{T}^3)$
satisfying
\begin{equation}
\int_{\mathbb{T}^3}|v(t,x)|^2dx=e(t).
\end{equation}
The starting point of their proof is to employ
a sequence $\{\bar u_{K_j}\}_j$ of the Euler-Reynolds system \eqref{averaged-NS}.
To construct a weak solution, we do not need to consider any coarse-graining, this means,
the stress tensor $\tau_K$ does not need to satisfy the formula $\overline{(u\otimes u)}_K-\bar u_K\otimes\bar u_K$.
Roughly saying,
 $\lim_{j\to\infty}\bar u_{K_j}$ ($K_0\ll K_1\ll K_2\ll\cdots$)
is their constructed weak solution, and
in this case, we need to show $\lim_{j\to\infty}\tau_{K_j}=0$
in a weak sense.
The key idea is to control $\tau_{K_j}$ by 
 $(\bar u_{K_{j+1}}-\bar u_{K_{j}})\otimes(\bar u_{K_{j+1}}-\bar u_{K_{j}})$, 
and in order to do so,  they employ a stationary Euler flow called Mikado flow.
In this procedure, the initial stress tensor $\tau_{K_0}$ becomes unknown, but the idea is  to replace it to the given function $e(t)$. 
Therefore we would emphasize that it must be interesting to find 
our vortex-stretching mechanism 
in each scale $1/K_j$ (with some appropriate coarse-graining) and this is our future work.
\end{remark}

The above theorems are  consequences of the following lemma for the 2D Euler equations, which was first established in a work of Bourgain and Li \cite{BL}. 
In this paper we sophisticate their idea, and systematically summarize the large Lagrangian deformation in the next section (the following lemma is 
a direct consequence of either Propositions \ref{BC} or \ref{prop:LLD-bubbles}).

\begin{lemma}[Instantaneous large Lagrangian deformation]\label{lem:LLD} 
	There exist a sequence of smooth initial data $\tilde{\omega}_{0,n}^L \in C^\infty_c(\mathbb{T}^2)$ normalized in $L^2(\mathbb{T}^2)$, $\delta > 0$, and sequences $M_n \rightarrow + \infty$, $t_n^* \rightarrow t^*_\infty\in [0,\delta]$, $r_n > 0$, and $x_n^* \in \mathbb{T}^2$ such that the 2D flow map $\eta_n$ associated with the solution for $\tilde{\omega}_{0,n}^L$ exhibits large Lagrangian deformation: \begin{equation}\label{eq:LLD-general}
	\begin{split}
	\inf_{ x \in B_{x_n^*}(r_n) } \left|D\eta_n(t_n^*, x) \right| \ge M_n
	\end{split}
	\end{equation} for all $n \ge 1$. 
\end{lemma}

\begin{proof}[Proof of Theorem \ref{thm:main} from Lemma \ref{lem:LLD}]

Using the above lemma,  we take $ t^*_n$, $x_n^\ast = (x^\ast_1, x^\ast_2)$ for which \begin{equation*}
\begin{split}
	\inf_{ x \in B_{x_n^*}(\delta_n) } \left|D\eta_n(t_n^*, x) \right| \ge M_n
\end{split}
\end{equation*} (here we assume, without loss of generality, $\partial_{1} \eta_{n,1}$ is large). Let us recall the small-scale vortex blob $\omega_{0,n}^S$, for each $n$.
Let $\rho_n\in C^\infty_{c}(\mathbb{T}^2)$ be 
such that $\|\nabla\rho_n\|_{L^2} = 1$ and $\text{supp}\, \rho_n\subset B(x^*_n,r_n)$. It is not difficult to ensure that $\nrm{\rho_n}_{L^2} \le 1$. 
We now choose the small-scale initial vorticity $\omega_{0,n}^S$ 
such that for $x_h=(x_1,x_2)$ 
\begin{equation*} 
u_{0,n}^S=(0,0,   
\rho_n(x_h-x_n^\ast))\quad\text{and}\quad 
\omega^S_{0,n}=\nabla\times u_{0,n}^S.
\end{equation*} 
A direct calculation yields (by arranging for instance that $\rho_n(x_h) = C_n x_2$ with $C_n$ depending only on $M_n$ and $\delta_n$ near $x_h = 0$)
\begin{equation}\label{norm-inflation estimate}
\|\omega_n^S(t^*_n)\|_{L^2}=
\|D\eta_n(t^*_n)\omega^S_{0,n}\|_{L^2} \gtrsim M_n\|\omega_{0,n}^S\|_{L^2} \approx M_n,
\end{equation} 
 and then this is the desired estimate.
\end{proof}
\begin{proof}[Proof of Theorem \ref{thm:main2}]
We note that for each $\nu > 0$, the solution $\{\omega^\nu_n\}$ 
exists globally-in-time, simply because the initial data is independent of $x_3$. 
For the sake of  convenience, we divide the proof into several subsections.

\subsection{Estimates on the solution}

We consider  initial data \begin{equation*}
\begin{split}
\omega_{0,n} = \omega_{0,n}^L + \omega_{0,n}^S
\end{split}
\end{equation*} where $\omega_{0,n}^L$ is smoothed Bahouri-Chemin at scale $2^{-n}$ and $\omega_{0,n}^S$ is some bump function supported at scale $2^{-n}$. The normalization is that $\nrm{ \omega_{0,n}^L}_{L^\infty} = 1$ and $\nrm{\omega_{0,n}^S}_{L^2} = 1$ for all $n$ (thus $\nrm{\omega_{0,n}^S}_{L^\infty}\sim 2^{n}$). 
We denote $\omega_n^L$ by the solution of 2D Euler with initial data $\omega_{0,n}^L$. Similarly, $\omega_n^{L,\nu}$ is the solution of 2D Navier-Stokes with viscosity $\nu > 0$ and with the same initial data. The corresponding velocity is denoted by $u_n^L$ and $u_n^{L,\nu}$, respectively.
On the other hand, we consider $u^S_n$, which is the solution of the transport equation 
\begin{equation*}
\begin{split}
\rd_t u^S_{n} + u^L_n \cdot \nabla u^S_n = 0 
\end{split}
\end{equation*} with initial data $u^S_{0,n} = -\nabla \times \lap^{-1}\omega^S_{0,n}$. Similarly, $u^{S,\nu}_n$ is defined by the solution of \begin{equation*}
\begin{split}
\rd_t u^{S,\nu}_{n} + u^{L,\nu}_n \cdot \nabla u^{S,\nu}_n = \nu\lap u^{S,\nu}_n 
\end{split}
\end{equation*} with the same initial data. 
Then, we have that \begin{equation*}
\begin{split}
\omega_n = \omega^L_n + \omega^S_n 
\end{split}
\end{equation*} is the solution of 3D Euler with initial data $\omega_{0,n}$ where $\omega^S_n = \nabla\times u_n^S$. Similarly, \begin{equation*}
\begin{split}
\omega_n^\nu = \omega^{L,\nu}_n + \omega^{S,\nu}_n
\end{split}
\end{equation*} is the solution of 3D Navier-Stokes with initial data $\omega_{0,n}$ with viscosity $\nu > 0$. Here $\omega^{S,\nu}_n = \nabla\times u^{S,\nu}_n$. 
We first estimate the Euler solution. From the maximum principle $\nrm{\omega^L_{n}(t)}_{L^\infty} = 1$ for all time and hence (see \eqref{log-exponential-estimate})
\begin{equation*}
\begin{split}
\nrm{\nabla u^L_n(t)}_{L^\infty} \lesssim \log \nrm{\omega_{0,n}^L}_{C^1} e^{Ct} \lesssim n 
\end{split}
\end{equation*} on the time interval $[0,1]$. Then we have, from the classical estimate
(see \cite{BKM} for example) 
\begin{equation*}
\begin{split}
\nrm{\omega_n^L(t)}_{H^s} \lesssim \nrm{\omega_{0,n}^L(t)}_{H^s} e^{C(s)\int_0^t \nrm{\nabla u^L_n(\tau)}_{L^\infty} d\tau} 
\end{split}
\end{equation*} that \begin{equation*}
\begin{split}
\nrm{\omega_n^L(t)}_{H^s} \lesssim 2^{c(s)n}  
\end{split}
\end{equation*} for some constant $c(s) > 0$ depending only on $s > 1$ for any $s$ and $t \in [0,1]$. 
We note that the Navier-Stokes solution satisfy the same bounds: \begin{equation*}
\begin{split}
\nrm{\omega_n^{L,\nu}(t)}_{H^s} \lesssim 2^{c(s)n}  
\end{split}
\end{equation*} with constant independent of $\nu > 0$. This is because we still have the maximum principle $\nrm{\omega^{L,\nu}_{n}(t)}_{L^\infty} \le 1$ for all $t \ge 0$ and the $H^s$ estimate holds \textit{a fortiori} for the Navier-Stokes. Taking $s>3$ and by the Sobolev embedding, we obtain
\begin{equation}\label{Sobolev-embedding}
\nrm{\nabla^2u_n^L(t)}_{L^\infty}+\nrm{\nabla^2u_n^{L,\nu}(t)}_{L^\infty} \lesssim 2^{cn}
\quad\text{for}\quad t\in[0,1]. 
\end{equation} 
This is elementary but is the key in the estimates below.
We now control $\omega_n^S$ and $\omega_n^{S,\nu}$. To begin with, we note that the corresponding velocities are stable in $L^\infty$ for all time:
 \begin{equation*}
\begin{split}
\nrm{u_n^S}_{L^\infty}, \nrm{u_n^{S,\nu}}_{L^\infty} \lesssim 1, 
\end{split}
\end{equation*} which is clear from the maximum principle. Next, from \begin{equation*}
\begin{split}
\rd_t \omega_n^S + (u_n^L\cdot\nabla)\omega_n^S = \nabla u_n^L \omega_n^S,
\end{split}
\end{equation*} we see that \begin{equation*}
\begin{split}
\nrm{\omega_n^S(t)}_{L^\infty} \le \nrm{\omega_{0,n}^S}_{L^\infty} e^{\int_0^t\nrm{ \nabla u_n^L(\tau)}_{L^\infty}d\tau} \lesssim 2^{cn}. 
\end{split}
\end{equation*} Similarly, from \begin{equation*}
\begin{split}
\rd_t \omega_n^{S,\nu} + (u_n^{L,\nu}\cdot\nabla)\omega_n^{S,\nu} = \nabla u_n^{L,\nu} \omega_n^{S,\nu} + \nu \lap \omega_n^{S,\nu},
\end{split}
\end{equation*} one can obtain $L^\infty$ estimates \begin{equation*}
\begin{split}
\nrm{\omega_n^{S,\nu}(t)}_{L^\infty} \lesssim 2^{cn}. 
\end{split}
\end{equation*}
We shall need just one more estimate: from \begin{equation*}
\begin{split}
\rd_t \nabla \omega_n^S + (u_n^L\cdot\nabla)\nabla\omega_n^S = 2\nabla u_n^L \nabla\omega_n^S + \nabla^2 u_n^L \omega_n^S, 
\end{split}
\end{equation*} we obtain \begin{equation*}
\begin{split}
\frac{d}{dt} \nrm{\nabla \omega_n^S }_{L^\infty} \lesssim n \nrm{\nabla \omega_n^S }_{L^\infty} + 2^{cn}  
\end{split}
\end{equation*} where we have used \eqref{Sobolev-embedding}
and 
$\nrm{\omega_n^S(t)}_{L^\infty} \lesssim 2^{cn}$. 
Together with 
$\nrm{\nabla \omega_{0,n}^S}_{L^\infty} \lesssim 2^{cn}$,
we conclude that \begin{equation*}
\begin{split}
\nrm{\nabla \omega_{n}^S(t)}_{L^\infty} \lesssim  2^{cn}
\end{split}
\end{equation*} on $t \in [0,1]$. 

\subsection{$L^2$ inviscid limit estimate on the large-scale velocity}

We compare the 2D Euler and Navier-Stokes equations  of the velocity:
  \begin{equation*}
\begin{split}
&\rd_t u_n^{L,\nu} + u_n^{L,\nu} \cdot \nabla u_n^{L,\nu} + \nabla p_n^{L,\nu} = \nu \lap u_n^{L,\nu} , \\
&\rd_t u_n^L + u_n^L\cdot\nabla u_n^L + \nabla p_n = 0. 
\end{split}
\end{equation*} 
Then, we see that 
 \begin{equation*}
\begin{split}
\frac{1}{2}\frac{d}{dt} \nrm{u_n^{L,\nu} - u_n^L}_{L^2}^2 + \int( u_n^L - u_n^{L,\nu}) \cdot \nabla u_n^L \cdot (u_n^{L,\nu} - u_n^L)  
= \nu \int \lap u_n^{L,\nu} \cdot (u_n^{L,\nu} - u_n^L) .
\end{split}
\end{equation*} We handle the right hand side as follows: \begin{equation*}
\begin{split}
-\nu\int |\nabla u_n^{L,\nu}|^2 + \nu \int \nabla u_n^{L,\nu} : \nabla u_n^L .
\end{split}
\end{equation*} Then, using the  Cauchy-Schwarz inequality and Young's inequality, we have \begin{equation*}
\begin{split}
\frac{d}{dt} \nrm{u_n^{L,\nu} - u_n^L}_{L^2}^2  \lesssim \nrm{\nabla u_n^L}_{L^\infty} \nrm{u_n^{L,\nu} - u_n^L}_{L^2}^2 + \nu \nrm{\nabla u_n^L}_{L^2}^2 \lesssim n\nrm{u_n^{L,\nu} - u_n^L}_{L^2}^2 + \nu ,
\end{split}
\end{equation*} which gives \begin{equation*}
\begin{split}
\nrm{u_n^{L,\nu} - u_n^L}_{L^2}^2 \lesssim \nu 2^{cn}
\end{split}
\end{equation*} for $t \in [0,1]$.

\subsection{$L^2$ inviscid limit estimate on the small-scale velocity}

We now compare the equations satisfied by $u_n^S$ and $u_n^{S,\nu}$: \begin{equation*}
\begin{split}
\rd_t u^S_{n} + u^L_n \cdot \nabla u^S_n = 0 
\end{split}
\end{equation*} with \begin{equation*}
\begin{split}
\rd_t u^{S,\nu}_{n} + u^{L,\nu}_n \cdot \nabla u^{S,\nu}_n = \nu\lap u^{S,\nu}_n . 
\end{split}
\end{equation*} Proceeding similarly as in the above, we obtain \begin{equation*}
\begin{split}
\frac{d}{dt} \nrm{u_n^{S,\nu} - u_n^S}_{L^2}^2  \lesssim \nrm{\nabla u_n^S}_{L^\infty} \nrm{u_n^{L,\nu} - u_n^L}_{L^2} \nrm{u_n^{S,\nu} - u_n^S}_{L^2} + \nu \nrm{\nabla  u_n^S}_{L^2}^2
\end{split}
\end{equation*} and inserting the previous estimates for $\nrm{u_n^{L,\nu} - u_n^L}_{L^2}$ and $\nrm{\nabla  u_n^S}_{L^2}$, \begin{equation*}
\begin{split}
\frac{d}{dt} \nrm{u_n^{S,\nu} - u_n^S}_{L^2}^2  \lesssim \nu( 2^{cn}\nrm{u_n^{S,\nu} - u_n^S}_{L^2} + 2^{cn}  ) \lesssim 2^{cn}\nu (1 + \nrm{u_n^{S,\nu} - u_n^S}_{L^2}).
\end{split}
\end{equation*} 
Here we figure out the relation between $\nu$ and $n$.
Taking $0 < \nu \le \nu_n := 2^{-2Mn}$ for some $M \gg 1$ (depending only on a few absolute constants), we can guarantee that \begin{equation*}
\begin{split}
\nrm{u_n^{S,\nu} - u_n^S}_{L^2} \lesssim 2^{-Mn}
\end{split}
\end{equation*} on $0 \le t \le 1$ for all $0 < \nu \le \nu_n$.

\subsection{$H^1$ inviscid limit estimate on the large-scale}

This time, we consider the 2D Navier-Stokes solutions in the vorticity form and compare it with the corresponding Euler solutions:
\begin{equation*}
\begin{split}
&\rd_t\omega_n^{L,\nu} + u_n^{L,\nu}\cdot\nabla \omega_n^{L,\nu} = \nu\lap \omega_n^{L,\nu},  \\
&\rd_t\omega_n^L + u_n^L\cdot\nabla \omega_n^L = 0 . 
\end{split}
\end{equation*} Then using previous bounds, \begin{equation*}
\begin{split}
\frac{d}{dt} \nrm{\omega_n^L - \omega_n^{L,\nu}}_{L^2}^2 &\lesssim \nrm{\nabla \omega_n^L}_{L^\infty} \nrm{u_n^L - u_n^{L,\nu}}_{L^2} \nrm{\omega_n^L - \omega_n^{L,\nu}}_{L^2}   + \nu \nrm{\nabla\omega_n^L}_{L^2}^2 \\
&\lesssim 2^{cn} 2^{-Mn} \nrm{\omega_n^L - \omega_n^{L,\nu}}_{L^2} + 2^{cn} 2^{-Mn} .
\end{split}
\end{equation*} This guarantees that \begin{equation*}
\begin{split}
\nrm{\nabla u_n^L -\nabla u_n^{L,\nu}}_{L^2} \sim \nrm{\omega_n^L - \omega_n^{L,\nu}}_{L^2} \lesssim 2^{-Mn}.
\end{split}
\end{equation*}

\subsection{$H^1$ inviscid limit estimate on the small-scale}

The goal is to show that \begin{equation*}
\begin{split}
\nrm{\omega_n^S - \omega_n^{S,\nu}}_{L^2}^2 \lesssim 1
\end{split}
\end{equation*} for all $0 < \nu \le \nu_n:= 2^{-2Mn}$ and $t \in [0,1]$ uniformly in $n$. 
Recall that
\begin{equation*}
\begin{split}
&\rd_t  \omega_n^{S,\nu} + (u_n^{L,\nu}\cdot\nabla)\omega_n^{S,\nu} = \nabla u_n^{L,\nu} \omega_n^{S,\nu}+\nu\Delta \omega_n^{S,\nu}, \\
&\rd_t  \omega_n^S + (u_n^L\cdot\nabla)\omega_n^S = \nabla u_n^L \omega_n^S. 
\end{split}
\end{equation*} 
We have \begin{equation*}
\begin{split}
&\frac{1}{2}\frac{d}{dt} \nrm{ \omega_n^S - \omega_n^{S,\nu}}_{L^2}^2 + \int (u_n^L - u_n^{L,\nu})\cdot\nabla\omega_n^S \cdot( \omega_n^S - \omega_n^{S,\nu}) \\
&\qquad = \int \nabla u_n^L  (\omega_n^S - \omega_n^{S,\nu} )\cdot ( \omega_n^S - \omega_n^{S,\nu}) + \int (\nabla u_n^L - \nabla u_n^{L,\nu}) \omega_n^{S,\nu}\cdot ( \omega_n^S - \omega_n^{S,\nu})   + \nu \int \lap \omega_n^{S,\nu}(\omega_n^{S,\nu} - \omega_n^S). 
\end{split}
\end{equation*} After some routine massaging, \begin{equation*}
\begin{split}
\frac{d}{dt} \nrm{ \omega_n^S - \omega_n^{S,\nu}}_{L^2}^2  & \lesssim \nrm{\nabla\omega_n^S}_{L^\infty} \nrm{ u_n^S - u_n^{S,\nu}}_{L^2}\nrm{ \omega_n^S - \omega_n^{S,\nu}}_{L^2}  + \nrm{\nabla u_n^L}_{L^\infty} \nrm{ \omega_n^S - \omega_n^{S,\nu}}_{L^2}^2 \\
&\qquad + \nrm{\omega_n^{S,\nu}}_{L^\infty} \nrm{\nabla u_n^L-\nabla u_n^{L,\nu}}_{L^2}\nrm{ \omega_n^S - \omega_n^{S,\nu}}_{L^2}  + \nu \nrm{\nabla \omega_n^S}_{L^2}^2 \\
& =: I + II + III + IV .
\end{split}
\end{equation*} We then just apply all the estimates that we have proven so far: \begin{equation*}
\begin{split}
|I| \lesssim 2^{cn} 2^{-Mn} \nrm{ \omega_n^S - \omega_n^{S,\nu}}_{L^2},
\end{split}
\end{equation*} \begin{equation*}
\begin{split}
|II| \lesssim n \nrm{ \omega_n^S - \omega_n^{S,\nu}}_{L^2}^2, 
\end{split}
\end{equation*} \begin{equation*}
\begin{split}
|III| \lesssim 2^{cn} 2^{-Mn} \nrm{ \omega_n^S - \omega_n^{S,\nu}}_{L^2},
\end{split}
\end{equation*} and \begin{equation*}
\begin{split}
|IV| \lesssim \nu 2^{cn} \lesssim 2^{-Mn},
\end{split}
\end{equation*} so that \begin{equation*}
\begin{split}
\frac{d}{dt} \nrm{ \omega_n^S - \omega_n^{S,\nu}}_{L^2}^2 \lesssim n\nrm{ \omega_n^S - \omega_n^{S,\nu}}_{L^2}^2 + 2^{-Mn}. 
\end{split}
\end{equation*} 
From the fact that $\nrm{ \omega_n^S - \omega_n^{S,\nu}}_{L^2}$ is initially 0, we conclude that $\nrm{ \omega_n^S - \omega_n^{S,\nu}}_{L^2} \lesssim 1$, by taking $M$ larger (still depending only on a few absolute constants) if necessary.

\subsection{Completion of the proof}
By Theorem \ref{thm:main} with Proposition \ref{BC}  and \eqref{norm-inflation estimate}, we see
\begin{equation*}
\begin{split}
\nrm{\omega_n^S}_{L^2} \gtrsim 2^{c_0n}
\end{split}
\end{equation*} for $t \in [\delta/2,\delta]$, with $0<\delta \le 1$.
Thus
\begin{equation*}
\begin{split}
\int_0^\delta\|\omega_n(t)\|_{L^2}^2dt  \gtrsim 2^{c_0n}.
\end{split}
\end{equation*} 
Then we finally have 
\begin{equation*}
\begin{split}
(\nu_n)^{a_0}\int_0^\delta\|\omega_n^\nu(t)\|_{L^2}^2dt &\gtrsim 
(\nu_n)^{a_0}\int_0^\delta\|\omega_n(t)\|_{L^2}^2dt
-(\nu_n)^{a_0} \delta \sup_{0<t<\delta}\|\omega_n^\nu(t)-\omega_n(t)\|_{L^2}^2\\
&
\gtrsim
2^{(c_0-2a_0M)n}-c 2^{-2a_0Mn}.
\end{split}
\end{equation*}
Now by taking $a_0=c_0/(2M)$, we obtain  the desired estimate.
\end{proof}

\section{Creation of Lagrangian deformation near the origin}\label{LD near the origin}

Here we consider the cases where occurrence of Lagrangian deformation on some ball centered at the origin can be established. In each setup we estimate the ball size as well. 

In the first scenario, we consider deformation by a single bubble (modulo odd-odd symmetry) which has a single length scale $\sim \ell$. We show deformation happens at a ball of size $\sim \ell $ near the origin, and we work out the multiplicative constant. 

Next, in the second case, we simply smooth out the Bahouri-Chemin data \begin{equation}\label{eq:BC}
\begin{split}
\omega_0 = \mathrm{sgn}(x_1)\mathrm{sgn}(x_2)
\end{split}
\end{equation} (introduced in \cite{BC}) at some length scale. In this situation, since almost all of the domain is occupied by vorticity, it is rather easy to achieve a (sharp) lower bound on the Lagrangian deformation near the origin, for time $O(1)$. 

Lastly, we consider the case of general ``bubbles'' (see \eqref{eq:bubbles}) and show that, as long as the sequence of coefficients is not summable, then short-time large Lagrangian deformation can be achieved near the origin. Here, a difficulty is that the bubbles tend to approach one of the axes quite rapidly (the rate of attraction increases as we put more bubbles near the origin), which slows down the stretching of the gradient of the flow map (in \cite{EJ}, a contradiction argument was used to overcome this difficulty). The existence of such Lagrangian deformation was first established in a celebrated work of Bourgain-Li (\cite{BL}), but our result is more quantitative, applicable for a more general class of initial data, and most importantly does not rely on a contradiction argument so that we know precisely where the large Lagrangian deformation is happening. 

\subsection{Lagrangian deformation from a single bubble}\label{single bubble}
 
We consider a bubble of length scale $\ell$, with odd-odd symmetry. To be precise, take a radial function $h$ on $\mathbb{R}^2$, defined by \begin{equation*}
\begin{split}
h(x) = \begin{cases}
1 & |x| \le 1 \\
2-|x| &  1 \le |x| \le 2 \\
0 & \mbox{otherwise}.
\end{cases}
\end{split}
\end{equation*} We may slightly mollify $h$ at the corners so that it is a smooth radial function with $\nrm{h}_{L^\infty} \le 1, \nrm{h'}_{L^\infty}\le 1 + 1/10$, and $\nrm{h''}_{L^\infty}\le 10$ (say). Then set 
\begin{equation}\label{piece-of-self-similar}
\begin{split}
\zeta(x_1,x_2) = \sum_{\varepsilon_1, \varepsilon_2 = \pm 1} h(x_1 -2\varepsilon_1, x_2 - 2\varepsilon_2)
\end{split}
\end{equation}
 and for any $\ell \le 1$, \begin{equation*}
\begin{split}
\omega_{0,\ell}(x) := \zeta(\ell^{-1}x) 
\end{split}
\end{equation*} defines a smooth vorticity which is supported away from the origin and whose support is contained in $B_0(4\ell)$. We have $$\nrm{\omega_{0,\ell}}_{C^{0,1}} := \sup_{x \ne x'} \frac{\left| \omega_{0,\ell}(x) - \omega_{0,\ell}(x')\right|}{|x-x'|} \le (1 + \frac{1}{10})\ell.$$

Let $\omega_\ell(t)$ be the solution with initial data $\omega_{0,\ell}$ and $u_\ell(t) = \nabla^\perp\Delta^{-1}\omega_{\ell}(t)$. At $t = 0$, we compute\footnote{
	Here we compute in $\mathbb{R}^2$ rather than $\mathbb{T}^2$ for simplicity, but all the constants will be approximately the same at least for $0 < \ell \ll 1$.)}  using explicit formulas \begin{equation*}
\begin{split}
\rd_1 u_{1,\ell}(0,x) &= \frac{1}{2\pi}\int_{\mathbb{R}^2} \frac{2(y_1-x_1)(y_2-x_2)}{|x-y|^4}\omega_{0,\ell}(y) dy, \\
\rd_1 u_{2,\ell}(0,x) &= \frac{1}{2\pi}\int_{\mathbb{R}^2} \frac{(y_1-x_1)^2-(y_2-x_2)^2}{|x-y|^4}\omega_{0,\ell}(y) dy
\end{split}
\end{equation*}  that $\rd_1 u_{1,\ell}(0,x) = -\rd_2 u_{2,\ell} > 1/2$ for $x \in [0,0.5\ell]^2$ whereas $|\rd_1 u_{2,\ell}|, |\rd_2 u_{1,\ell}| < 1/20$ in the same region. From the smoothness of the Euler solution in time and space, it follows that these inequalities hold on the same ball for some interval of time. Then, using the system \eqref{eq:stretch-ODE} it is not difficult to show that at least during that time interval, stretching of the gradient $\rd_1\eta_1(t,x)$ occurs on $[0,0.5\ell]^2$.

This shows that vorticity supported on a ball of radius $\approx 2\ell$ (with odd symmetry) is able to stretch vorticity in a ball of radius $\approx 0.5\ell$ and pointing in an orthogonal direction, which is comparable with the third conclusion in GSK  (see also \cite[Section D]{GSK});
that is, vortices are most likely to be stretched in strain fields around $2$-$8$ times larger vortices. In  Appendix, we derive Kolmogorov's $-5/3$-law from 
GSK-type of self-similar hierarchy, and  ideally, \eqref{piece-of-self-similar} is the minimum piece of it.

\subsection{Large Lagrangian deformation by smoothed Bahouri-Chemin}
Here we show large Lagrangian deformation by smoothing the Bahouri-Chemin stationary solution. Recalling our convention that $\mathbb{T}^2 = \mathbb{R}^2/(2\mathbb{Z})^2$, the Bahouri-Chemin solution can be written as $\mathrm{sgn}(x_1)\mathrm{sgn}(x_2)$ where $|x_1|, |x_2| \le 1$. Given $n \ge 1$, we cut it near the axes as follows: \begin{equation*}
\begin{split}
\tilde{\omega}_n := \mathrm{sgn}(x_1)\mathrm{sgn}(x_2) \chi_{ \{ 2^{-n} < |x_1| , |x_2|< 1 - 2^{-n} \} } 
\end{split}
\end{equation*}
Now let $\varphi \in C^\infty_c(\mathbb{R}^2)$ be a standard mollifier;  a radial function whose support is contained in the unit ball. With $\varphi_{\ell}(x) := \ell^{-2}\varphi(\ell^{-1}x)$, we define 
\begin{equation}\label{BC-initial-data}
\omega_{0,n}^L := \varphi_{2^{-n-1}} * \tilde{\omega}_n.
\end{equation}

We recall a simple estimate of Yudovich (see e.g. \cite{EJ} for a proof): \begin{lemma*}
	Let $\omega(t) \in L^\infty([0,\infty): L^\infty(\mathbb{T}^2))$ be a solution of the 2D Euler equations on $\mathbb{T}^2$, and $\eta$ be the associated flow map. Then for some absolute constant $c > 0$, we have \begin{equation}\label{eq:Yud}
	\begin{split}
	|x-x'|^{1+ct\nrm{\omega_0}_{  L^\infty}} \le |\eta(t,x) - \eta(t,x')| \le |x-x'|^{1-ct\nrm{\omega_0}_{  L^\infty}},
	\end{split}
	\end{equation} for all $0 \le t \le 1$ and $|x-x'| \le 1/2$. 
\end{lemma*}

\begin{prop}\label{BC}
	For any integer $n \ge 1$, consider the smoothed (at scale $2^{-n}$) Bahouri-Chemin initial data $\omega_{0,n}^L$. Then there exists a constant $\delta> 0$ such that for any $0 < t^*  \le \delta$, we have \begin{equation}\label{eq:LLD}
	\begin{split}
	\inf_{ |x| \le  2^{-(1+\mathfrak{c}_0t^*){n}} }\nrm{ D\eta(t^* , x) }_{L^\infty} \gtrsim \exp(cnt^* )
	\end{split}
	\end{equation} for some absolute constant $\mathfrak{c}_0 > 0$. That is, large Lagrangian deformation happens on a ball of radius $\sim 2^{-(1+\mathfrak{c}_0t^*)n}$ for some time. 
\end{prop}

\begin{remark}
	The estimate \eqref{eq:LLD} is sharp. To see this, recall the standard   estimate for the 2D Euler equations (see e.g. \cite[Theorem 2.1, Proposition 2.2]{KS}) 
\begin{equation}\label{log-exponential-estimate}
	\begin{split}
	\nrm{\nabla u(t)}_{L^\infty}  \lesssim \nrm{\omega_0}_{L^\infty}\left(1+ \log\frac{\nrm{\omega_0}_{C^1}}{\nrm{\omega_0}_{L^\infty}}\right)  \exp(C\nrm{\omega_0}_{L^\infty}t). 
	\end{split}
	\end{equation} 
{Strictly speaking, in \cite{KS}, the bound is stated in terms of the $C^\alpha$-norm of the vorticity with $0 < \alpha <1$, but one may simply pick some $0 < \alpha <1$ and use $\nrm{\omega_0}_{C^\alpha} \lesssim \nrm{\omega_0}_{C^1}$.}
Applying this to the initial data in the above gives $\nrm{\nabla u(t)}_{L^\infty} \lesssim n$, on $t \in [0,1]$, {simply because $\nrm{\omega_0}_{C^1} \lesssim 2^{cn}$ for some $c>0$.} Hence we have the upper bound \begin{equation*}
	\begin{split}
	\nrm{D\eta(t)}_{L^\infty} \le \exp\left( 2\int_0^t \nrm{\nabla u(s)}_{L^\infty} ds \right) \le \exp\left( Cnt \right). 
	\end{split}
	\end{equation*}
For the upper bound on more general setting (general dimension), we refer Theorem 2.1 and  Proposition 2.3 in \cite{CD}.
\end{remark}

\begin{proof}
	We note that $\omega^L_{0,n}$ is odd with respect to both axis, $\omega^L_{0,n} = 1$ on $[2^{-n+1},1-2^{-n+1}]^2 $, and vanishes on $([0,1] \backslash [2^{-n-1},1-2^{-n-1}])^2$. From now on, we fix some large $n$ and omit the dependence of $\omega_n$ and $\eta_n$ in $n$ and simply write $\omega$ and $\eta$. We claim that for some absolute constant $\delta > 0$, \begin{equation*}
	\begin{split}
	\omega_n^L(t,x) \equiv 1 \quad \mbox{on}\quad [0,\delta]\times [2^{-n/2},1/2]^2. 
	\end{split}
	\end{equation*} To show this, it suffices to observe that fluid particles starting from $\partial([2^{-n+1},1-2^{-n+1}]^2 )$ cannot reach the internal square $(2^{-n/2},1/2)^2$ within time $\delta$. For this we need to consider four sides of $\partial([2^{-n+1},1-2^{-n+1}]^2 )$. We only consider the left side, as the other sides can be treated in a similar way. To this end take a point of the form $x = (2^{-n+1},a)$ for some $2^{-n+1} \le a \le 1-2^{-n+1}$. Setting $x' = (0,a)$ and applying \eqref{eq:Yud}, we obtain \begin{equation*}
	\begin{split}
	|\eta_1(t,x) - \eta_1(t,x')| \le 2^{(-n+1)(1-ct)} \le 2^{-n/2}
	\end{split}
	\end{equation*} for $t \le \delta := 1/(2c)$, since $\nrm{\omega_0}_{ L^\infty} = 1$. Since $\eta_1(t,x') = 0$  and $\eta_1(t,x) > 0$ for all $t$, we deduce that $\eta_1(t,x) \le 2^{-n/2}$.
	
	From now on we shall restrict to $t \in [0,\delta]$. Using explicit formulas \begin{equation*}
	\begin{split}
	\rd_1 u_1(t,0) = \frac{4}{\pi} \int_{(\mathbb{R}_+)^2} \frac{y_1y_2}{|y|^4} \omega(y) dy = - \rd_2 u_2(t,0) 
	\end{split}
	\end{equation*} (where we have extended $\omega$ to $\mathbb{R}^2$ by periodicity), we obtain a lower bound \begin{equation}\label{eq:lb}
	\begin{split}
	\rd_1 u_1(t,0) = - \rd_2 u_2(t,0) \ge c_0 n 
	\end{split}
	\end{equation} for some $c_0 > 0$ with $n$ sufficiently large. {To see this lower bound, we write \begin{equation*}
		\begin{split}
		\int_{(\mathbb{R}_+)^2} \frac{y_1y_2}{|y|^4} \omega(y) dy  = \int_{[0,1]^2} \frac{y_1y_2}{|y|^4} \omega(y) dy  + \int_{(\mathbb{R}_+)^2\backslash[0,1]^2} \frac{y_1y_2}{|y|^4} \omega(y) dy ,
		\end{split}
		\end{equation*} and one can note that while the first integral can be evaluated to satisfy $\gtrsim n$, the second integral is bounded uniformly in $n$, after using that the integral of vorticity over $[n_1-1,n_1+1]\times [n_2-1,n_2+1] = 0$ for any $n_1,n_2$ (see for instance \cite{Z}).}  On the other hand, $\rd_1 u_2(t,0) = \rd_2 u_1(t,0) = 0$ by odd symmetry.

In order to estimate $\partial_1u_1$ and $\partial_2u_1$ not only at the origin but also in a small ball, we shall use the estimates for the 2D Euler solutions: with $\nrm{\omega_0}_{L^\infty} = 1$, it is well-known that   \begin{equation*}
	\begin{split}
	1 + \log\left( 1  + \nrm{\omega(t)}_{C^{1}} \right) \le  \left( 1 + \log\left( 1  + \nrm{\omega_0}_{C^{1}} \right) \right) \exp(Ct) 
	\end{split}
	\end{equation*} (cf. \cite[Theorem 2.1]{KS}). Now note that $\nrm{\omega_0}_{C^{1}} \approx 2^{n}$. By $ \exp(Ct)\lesssim 1+ct$ ($0<t<\delta$), we therefore obtain \begin{equation*}
	\begin{split}
	\nrm{\omega(t)}_{C^1} \lesssim 2^{(1+ct)n},\qquad t \in [0,\delta],
	\end{split}
	\end{equation*} by taking $\delta$ smaller if necessary. Then we use \begin{equation*}
	\begin{split}
	\nrm{\nabla u(t)}_{C^{\frac{1}{2}}} \lesssim \nrm{\omega(t)}_{C^{\frac{1}{2}}} \lesssim \nrm{\omega(t)}_{C^1}^{\frac{1}{2}} \lesssim 2^{\frac{1+ct}{2}n}.
	\end{split}
	\end{equation*} We then obtain \begin{equation*}
	\begin{split}
	|\nabla u(t,x) - \nabla u(t,x')| \lesssim 2^{\frac{1+ct}{2}n}|x-x'|^{\frac{1}{2}} <  \frac{c_0}{10}n , 
	\end{split}
	\end{equation*} for $|x-x'|<2^{-(1+ct)n}$ and $n$ sufficiently large, where $c_0 > 0$ is from \eqref{eq:lb}. Applying the above with $x' = 0$, we conclude the following:  
	 \begin{equation}\label{eq:lower}
	\begin{split}
	\inf_{ x \in (\mathbb{R}_+)^2 \cap B(0,2^{-(1+ct)n} ) } \inf_{t \in [0,\delta]} \rd_1 u_1(t,x) \ge \frac{9c_0}{10}n,
	\end{split}
	\end{equation} \begin{equation}\label{eq:upper}
	\begin{split}
	\sup_{ x \in (\mathbb{R}_+)^2 \cap B(0,2^{-(1+ct)n}) } \sup_{t \in [0,\delta]} |\rd_2 u_1(t,x) | \le  \frac{c_0}{10}n. 
	\end{split}
	\end{equation} 

	Now we consider the following system of ODEs: for each $x $,  denoting for simplicity $\eta := \eta(t,x)$ and $\rd_i\eta_j(t):= \rd_i\eta_j(t,x)$, \begin{equation}\label{eq:stretch-ODE}
	\begin{split}
	\frac{d}{dt} \rd_1 \eta_1(t) &= \rd_1 u_1(t,\eta ) \rd_1\eta_1(t) + \rd_2 u_1(t,\eta ) \rd_1\eta_2(t) \\
	\frac{d}{dt} \rd_1 \eta_2(t) &= -\rd_1 u_1(t,\eta ) \rd_1\eta_2(t) + (\rd_2 u_1(t,\eta ) +\omega(t,\eta))  \rd_1\eta_1(t) .
	\end{split}
	\end{equation} We claim that if $x$ satisfies \begin{equation}\label{eq:bound-condition}
	\begin{split}
	|\eta(t,x)| \le  2^{-(1+ct)n} 
	\end{split}
	\end{equation} for all $t \in [0,\delta]$, then \begin{equation*}
	\begin{split}
	\rd_1\eta_1(t,x) \ge \exp\left( \frac{c_0}{10}nt \right),\quad \forall t \in [0,\delta]. 
	\end{split}
	\end{equation*} This can be proved by observing that we have, on the same time interval, $\rd_1\eta_1(t) > 4|\rd_1\eta_2(t)|$ which follows from a continuity argument in time and the bounds \eqref{eq:upper}, \eqref{eq:lower}.   {For this, fix some $t > 0$ and consider two cases: $\rd_1\eta_2(t) \ge 0$ and $\rd_1\eta_2(t) < 0$. In the first case, after some rearranging we have \begin{equation*}
		\begin{split}
		\frac{d}{dt}\left( \rd_1\eta_1(t) - 4|\rd_1\eta_2(t)| \right)& = \rd_1u_1(\rd_1\eta_1 + 4\rd_1\eta_2) + \rd_2u_1 \rd_1\eta_2 - 4(\rd_2u_1+\omega)\rd_1\eta_1 \\
		& > (\rd_1u_1-4 (|\rd_2 u_1| + 1))(\rd_1\eta_1 + 4\rd_1\eta_2) + 4(|\rd_2u_1|+1)\rd_1\eta_1 \\
		&\quad + 16|\rd_2 u_1|\rd_1\eta_2 - |\rd_2u_1|\rd_1\eta_2 - 4(|\rd_2u_1|+1) \rd_1\eta_1 > 0. 
		\end{split}
		\end{equation*} In the other case, we instead have \begin{equation*}
		\begin{split}
		\frac{d}{dt}\left( \rd_1\eta_1(t) - 4|\rd_1\eta_2(t)| \right)& > 
		 (\rd_1u_1-4 (|\rd_2 u_1| + 1))(\rd_1\eta_1 + 4|\rd_1\eta_2|)  + 4(|\rd_2u_1|+1)\rd_1\eta_1 \\
		 & \quad + 16|\rd_2u_1||\rd_1\eta_2| - 4(|\rd_2u_1|+1)\rd_1\eta_1 >0 . 
		\end{split}
		\end{equation*} In the above estimates, we have used that $\rd_1\eta_1 > 0$, which follows again with a continuity argument in time and observing the bound \begin{equation*}
		\begin{split}
		\frac{d}{dt}\rd_1\eta_1 > (\rd_1u_1-\frac{1}{4}|\rd_2u_1|)\rd_1\eta_1 + \frac{1}{4}|\rd_2u_1|(\rd_1\eta_1 - 4|\rd_1\eta_2|),
		\end{split}
		\end{equation*} which only increases as long as $\rd_1\eta_1 - 4|\rd_1\eta_2| \ge 0$. 
	}
More precisely, even if there is $t$ such that $\partial_1\eta_2(t)=0$, that is,
$\partial_1\eta_1(t)-4|\partial_1\eta_2(t)|=0$, then $\partial_t\partial_1\eta_1(t)>0$.
This means $\partial_1\eta_1(t)>0$ for all $t>0$.
In turn, $\rd_1\eta_1(t) > 4|\rd_1\eta_2(t)|$ implies $|\rd_2 u_1(t,\eta ) \rd_1\eta_2(t) | < \rd_1 u_1(t,\eta ) \rd_1\eta_1(t)/4$, so that \begin{equation*}
	\begin{split}
	\frac{d}{dt} \rd_1 \eta_1(t) \ge \frac{3}{4} \rd_1 u_1(t,\eta ) \rd_1\eta_1(t)
	\end{split}
	\end{equation*} which gives the claim. It only remains to check the condition \eqref{eq:bound-condition}. Let $t^* \le \delta$ and using \eqref{eq:Yud} again, \begin{equation*}
	\begin{split}
	|\eta(t^*,x)| \le |x|^{1 - ct^* } \le  2^{-(1+ct)n} 
	\end{split}
	\end{equation*} for all sufficiently large $n$ if $|x| \le  2^{-(1+2ct^*)n}$.  This finishes the proof. 
\end{proof}

\subsection{Large Lagrangian deformation by Bourgain-Li bubbles}

We now recall the construction of odd-odd ``bubbles'' as in \cite{BL}. Given any smooth radial bump function $0 \leq \phi \leq 1$ with support in the ball $B(0,1/4)$ 
let 
\begin{align*} 
\phi_0(x_1, x_2) 
= 
\sum_{\varepsilon_1, \varepsilon_2 = \pm 1} 
\varepsilon_1 \varepsilon_2 \phi(x_1 {-} \varepsilon_1, x_2 {-} \varepsilon_2). 
\end{align*} 
We assume further that $\phi = 1$ in the ball $B(0,1/8)$. Clearly, the function $\phi_0$ is odd with respect to both $x_1$ and $x_2$. 
Define 
\begin{align}\label{eq:bubbles} 
\omega_{0,n}^L(x)  
=  
\sum_{k=1}^{n} a_k \phi_k( x) 
\end{align}  for some bounded sequence of non-negative numbers $\{a_k \}$, 
where 
$\phi_k(x) =  \phi_0 (2^k x)$.
Note that the supports of $\phi_k$ are disjoint and compact. Without loss of generality, we may assume that $\sup_{k=1}^{\infty} a_k \le 1$, so that $\nrm{\omega_{0,n}}_{L^\infty} \le 1$ uniformly in $n$. The further regularity of $\omega_{0,n}^L$ depends upon the asymptotic behavior of the coefficients $\{ a_k \}$ as $k \rightarrow \infty$. In the $L^\infty$-normalized case, i.e. \begin{equation}\label{eq:bubbles-Linfty}
\begin{split}
\omega_{0,n}^L(x)  = \sum_{k=1}^{n}  \phi_0(2^k x) , 
\end{split}
\end{equation} we have $\omega_{0,n}^{L} \in H^s \cap L^\infty$ uniformly in $n$ for all $s < 1$. When $a_k = k^{-1/2-\eps}$ for some $\eps>0$, $\omega_{0,n}^L \in H^1$ uniformly in $n$, and this specific choice of coefficients was utilized in \cite{BL} to show ill-posedness of the 2D Euler equations in $H^1$. 

\begin{prop}\label{prop:LLD-bubbles}
	Let $\{a_k\}_{k=1}^\infty$ be a bounded sequence of non-negative reals, and  $\omega_n^L(t) $ be the solution with initial data as in \eqref{eq:bubbles} with associated Lagrangian flow $\eta_n$. Set $S_k := S_{k-1} + a_{k}$ for $k \ge 1$ with $S_0 := 1$. Then, for some absolute constant $c_0 > 0$, we have \begin{equation}\label{eq:LLD-bubbles}
	\begin{split}
	|D\eta_n(t^*,0)|  \gtrsim (S_nt^*)^{c_0}
	\end{split}
	\end{equation}  for all $t^* \in [0,1]$ and $n$ sufficiently large. In particular, assuming $S_n$ is divergent in $n$, we have \begin{equation}\label{eq:LLD-bubbles-shorttime}
	\begin{split}
	|D\eta_n(S_n^{-\frac{1}{2}},0)| \gtrsim S_n^{\frac{c_0}{2}} 
	\end{split}
	\end{equation} again for all sufficiently large $n$. 
\end{prop} 

\begin{proof}
	To begin with, we shall fix some $n$ large and write $\eta = \eta_n$ for simplicity. Moreover, we note that we may assume $S_n$ is divergent in $n$, since otherwise \eqref{eq:LLD-bubbles} trivially holds for $t^* \in [0,1]$ as $|D\eta(t^*,0)| \ge 1$ always. We systematically use the following ``Key Lemma'' due to Kiselev and Sverak \cite{KS} (more precisely, a version on $\mathbb{T}^2$ from \cite{Z}), written in the polar coordinates 
 for convenience: 
 
\noindent \textbf{Key Lemma}. \textit{ Assume the vorticity on $\mathbb{T}^2$ is bounded and odd with respect to both axis. Then $u = \nabla^\perp\Delta^{-1}\omega$ satisfies \begin{equation}\label{eq:keyLemma}
	\begin{split}
	u(t,r,\theta) = \begin{pmatrix}
	\cos\theta \\
	-\sin\theta
	\end{pmatrix} r I(t,r) + r B(t,r,\theta)
	\end{split}
	\end{equation} for $|r| \le 1/2$, where \begin{equation*}
	\begin{split}
	I(t,r) := \frac{4}{\pi}\int_0^{\pi/2}\int_{  2r}^{1 } \frac{\sin(2\theta')}{s} \omega(t,s,\theta') dsd\theta'
	\end{split}
	\end{equation*} and \begin{equation*}
	\begin{split}
	\nrm{B(t)}_{L^\infty} \le C \nrm{\omega(t)}_{L^\infty}
	\end{split}
	\end{equation*} for some absolute constant $C > 0$. }
	
	For simplicity, we shall set $I(t) := I(t,0)$ which is well-defined as the vorticities we consider is always vanishing in a small neighborhood of the origin. We now give a brief outline of the argument. The goal is to estimate the time integration of the ``key integral'' $I(t)$, since then we deduce \begin{equation*}
	\begin{split}
	|D\eta(t^*,0)| \gtrsim \exp\left( c \int_0^{t^*} I(t) dt \right) . 
	\end{split}
	\end{equation*}  In turn, we may write \begin{equation*}
	\begin{split}
	I(t) = \sum_{ k = 1}^n I_k(t)
	\end{split}
	\end{equation*} where \begin{equation*}
	\begin{split}
	I_k (t) := \frac{4}{\pi} \int_0^{\pi/2} \int_0^{1/2} \frac{\sin(2\theta)}{r} a_k \phi_k(\eta^{-1}(t))(r,\theta) drd\theta 
	\end{split}
	\end{equation*} is the contribution to $I(t)$ from the bubble initially located at the scale $2^{-k}$. 
	
	Our strategy is to establish the following assertion  inductively in $k$: The  ``shape'' of the $k$-th bubble essentially remains the same within the time scale $t_k := c_1/S_k$ for some absolute constant $c_1 > 0$. Here what we mean by shape will be made precise below. Assuming for a moment that this statement holds,  we obtain that \begin{equation*}
	\begin{split}
	\int_0^{t_k} I_k(t) dt \gtrsim t_k I_k(0) \gtrsim \frac{a_k}{S_k}.
	\end{split}
	\end{equation*}  Then we have  \begin{equation*}
	\begin{split}
	\int_0^{t^*} I(t) dt \ge \sum_{k=1}^{n}\int_0^{\min\{t_k,t^*\}} I_k(t) dt \gtrsim \sum_{1\le k \le n, c_1 < t^*S_k} \frac{a_k}{S_k} 
	\end{split}
	\end{equation*} owing to the non-negativity of each $I_k(t)$. We now observe that, by approximating the sum with a Riemann integral, \begin{equation*}
	\begin{split}
	\sum_{k=1}^n \frac{a_k}{S_k} \approx \log(S_n) 
	\end{split}
	\end{equation*} and taking $k^*$ be the smallest number satisfying \begin{equation*}
	\begin{split}
	S_{k^*}t^* > c_1,
	\end{split}
	\end{equation*} (which exists by taking $n$ larger if necessary, since $t^* > 0$ and we are assuming that the sequence $S_{k^*}$ is divergent) \begin{equation*}
	\begin{split}
	\sum_{k=1}^{k^*} \frac{a_k}{S_k} \approx \log(S_{k^*}) \approx \log(c_1/t^*). 
	\end{split}
	\end{equation*} This gives \eqref{eq:LLD-bubbles}  and \eqref{eq:LLD-bubbles-shorttime} follows by taking $t^* = S_n^{-\frac{1}{2}}$. Hence it is sufficient to prove that \begin{equation*}
	\begin{split}
	I_k(t) \gtrsim I_k(0),\qquad t \in [0,t_k]
	\end{split}
	\end{equation*} uniformly in $k$. Below we shall formulate and prove a claim which implies the above lower bound.

	\medskip
	
	\textit{Step I: some preparations}
	
	\medskip
	
	We make some simple observations regarding the evolution of the bubbles. Recall from the definition of $\phi_0$ that restricted on to the positive quadrant, there exist ``rectangles'' $$\overline{R}_0 = \left\{ (r,\theta) : \overline{r}_1 < r < \overline{r}_2, \overline{\theta}_1 < \theta < \overline{\theta}_2 \right\}$$ and $$\underline{R}_0= \left\{ (r,\theta) : \underline{r}_1 < r < \underline{r}_2, \underline{\theta}_1 < \theta < \underline{\theta}_2 \right\}$$ such that \begin{equation*}
	\begin{split}
	\phi_0 = 1 \mbox{ on }  \underline{R}_0 \quad \mbox{ and } \quad \phi_0 = 0  \mbox{ outside of }  \overline{R}_0. 
	\end{split}
	\end{equation*} We may set \begin{equation*}
	\begin{split}
	\frac{1}{2} < \overline{r}_1 < \underline{r}_1 < \underline{r}_2 < \overline{r}_2 < 2 
	\end{split}
	\end{equation*} and \begin{equation*}
	\begin{split}
	\frac{\pi}{6} < \overline{\theta}_1 < \underline{\theta}_1 < \underline{\theta}_2 < \overline{\theta}_2 < \frac{\pi}{3} . 
	\end{split}
	\end{equation*} Now by simple scaling, with the $2^{-k}$-scaled rectangles $\underline{R}_k$ and $\overline{R}_k$, we have \begin{equation*}
	\begin{split}
	\phi_k = 1 \mbox{ on }  \underline{R}_k \quad \mbox{ and } \quad \phi_k = 0  \mbox{ outside of }  \overline{R}_k, 
	\end{split}
	\end{equation*} still restricted on the first quadrant (more precisely on $[0,1]^2$). This time, take an even smaller rectangle: \begin{equation*}
	\begin{split}
	\underline{R}_0^* = \left\{ (r,\theta) : \underline{r}_1^* < r < \underline{r}_2^*, \underline{\theta}_1^* < \theta < \underline{\theta}_2^* \right\} \subset \underline{R}_0
	\end{split}
	\end{equation*} where we may set \begin{equation*}
	\begin{split}
	\underline{r}_1^* = \frac{2\underline{r}_1 + \underline{r}_2}{3}, \quad \underline{r}_2^* = \frac{ \underline{r}_1 + 2\underline{r}_2}{3}
	\end{split}
	\end{equation*} and similarly \begin{equation*}
	\begin{split}
	\underline{\theta}_1^* = \frac{2\underline{\theta}_1 + \underline{\theta}_2}{3}, \quad \underline{\theta}_2^* = \frac{ \underline{\theta}_1 + 2\underline{\theta}_2}{3}. 
	\end{split}
	\end{equation*} Then as before define \begin{equation*}
	\begin{split}
	\underline{R}_k^* := \left\{ (r,\theta) : \underline{r}_1^* < 2^kr < \underline{r}_2^*, \underline{\theta}_1^* < \theta < \underline{\theta}_2^* \right\}. 
	\end{split}
	\end{equation*} Moreover, define \begin{equation*}
	\begin{split}
	\overline{A}_k := \{ (r,\theta) : 2^{-k-1} < r < 2^{1-k} , 0 < \theta < \frac{\pi}{2}  \}. 
	\end{split}
	\end{equation*} We shall now prove the following
	
	\smallskip
	
	\textbf{Claim.}  In the time interval $[0,t_k]$, the $k$-th bubble remains $a_k$ on the rectangle $\underline{R}_k^*$ and vanishes outside $\overline{A}_k$. Here $t_k := c_1/S_k$ with $c_1 > 0$ independent of $k$.
	
	\smallskip	
	
	\noindent This is what we mean by retaining the same ``shape''.
	We now rewrite the evolution of the trajectories in polar coordinates, using \eqref{eq:keyLemma}. Given some $x \in [0,1]^2$, we shall express the point $\eta(t,x)$ using $|\eta|$ and $\theta(\eta)$. Then, \begin{equation}\label{eq:traj-polar-rad}
	\begin{split}
	\frac{d}{dt}|\eta| & = u(t,\eta) \cdot \begin{pmatrix}
	\cos(\theta(\eta)) \\
	\sin(\theta(\eta)) 
	\end{pmatrix} \\
	& = |\eta| \left( \cos(2\theta(\eta)) I(t,|\eta|) + (\cos(\theta( \eta))B_1 + \sin(\theta(\eta))B_2) \right) 
	\end{split}
	\end{equation} and \begin{equation}\label{eq:traj-polar-angle}
	\begin{split}
	\frac{d}{dt}\theta(\eta) & = \frac{u(t,\eta)}{|\eta|} \cdot \begin{pmatrix}
	-\sin(\theta(\eta)) \\
	\cos(\theta(\eta)) 
	\end{pmatrix} \\
	& = -\sin(2\theta(\eta))  I(t,|\eta|) +  (-\sin(\theta( \eta))B_1 + \cos(\theta(\eta))B_2)
	\end{split}
	\end{equation} where $B = (B_1, B_2)$ is from \eqref{eq:keyLemma}.

	\medskip
	
	\textit{Step II: induction base case $k = 1$}
	
	\medskip
	We shall be concerned with the bubble $\phi_1$ and the trajectories $\eta(t,x)$ where $x \in \supp(\phi_1)$. Using the Yudovich estimate \eqref{eq:Yud} with $x' = 0$, we see that such trajectories are trapped inside the region $\{ 2^{-2} \le r \}$ during $[0,t_1]$ by choosing $c_1 > 0$ depending only on  $c\nrm{\omega_0}_{L^\infty}$ in \eqref{eq:Yud}. Similarly, trajectories starting from $\cup_{k>1}\supp(\phi_k)$ cannot cross the circle $\{ r = 2^{-2}\}$. This results in a naive bound \begin{equation*}
	\begin{split}
	I_1(t,|\eta|) \le I_1(t,2^{-2}) \lesssim a_1 
	\end{split}
	\end{equation*} on the same time interval. We use this to obtain slightly improved estimates on $|\eta|$ in \eqref{eq:traj-polar-rad}: \begin{equation*}
	\begin{split}
	\left| \frac{d}{dt} \ln \frac{1}{|\eta|} \right| \lesssim S_1  
	\end{split}
	\end{equation*} using $|B| \lesssim 1$. This guarantees that, given any small $\epsilon > 0$, by taking $c_1 = c_1(\epsilon) > 0$  small enough if necessary, we have \begin{equation*} 
	\begin{split}
	\left|\ln \frac{|\eta(0,x)|}{|\eta(t,x)|}\right| < \epsilon, \quad t\in [0,t_1],
	\end{split}
	\end{equation*}  recalling that $t_1 = c_1/S_1$. For the angle, we simply use \begin{equation*}
	\begin{split}
	\left| \frac{d}{dt} \theta(\eta)  \right| \lesssim S_1 
	\end{split}
	\end{equation*} to deduce \begin{equation*}
	\begin{split}
	\left| \theta(\eta(t,x))  - \theta(\eta(0,x)) \right| < \epsilon
	\end{split}
	\end{equation*} again for $t\in [0,t_1]$ by taking $c_1 > 0$ smaller if necessary. Thus, a suitable choice of $\epsilon > 0 $ (depending only on $\underline{r}_{1}, \underline{r}_{2}, \underline{\theta}_1, \underline{\theta}_2$) finishes the proof of the \textbf{Claim} for the case $k = 1$.
		
	\medskip
	
	\textit{Step III: completing the induction}
	
	\medskip
	We now assume that for some $k_0 > 1$ the \textbf{Claim} has been proved for $k = 1, \cdots, k_0-1$. We are now concerned with the trajectories $\eta(t,x)$ where $t \le t_{k_0}$ and $x \in \supp(\phi_{k_0})$. The induction hypothesis guarantees that, as long as $2^{-(k_0 + 1)} < |\eta| < 2^{-(k_0 -1)}$, we have that \begin{equation*}
	\begin{split}
	\left| \frac{d}{dt} \ln \frac{1}{|\eta|} \right| \lesssim S_{k_0} 
	\end{split}
	\end{equation*}  {simply because $t_k$ is decreasing with $k$ and the hypothesis ensures that the contribution of $a_k\phi_k \circ \eta^{-1}(t)$ to the key integral is bounded by $c a_k$ with some $c$ independent of $k$, for $k = 1, \cdots, k_0 - 1$. 
Strictly saying, here, we use the even smaller rectangles
$\underline{R}^*_k$. Thus $c$ is depending on $\epsilon>0$. 
This implies that 
 \begin{equation*} 
	\begin{split}
	\left|\ln \frac{|\eta(0,x)|}{|\eta(t,x)|}\right| < \epsilon , \quad t\in [0,t_{k_0}]
	\end{split}
	\end{equation*} for the \textit{same} $\epsilon$ and $c_1$.} Similarly, we can deduce \begin{equation*}
	\begin{split}
	\left| \theta(\eta(t,x))  - \theta(\eta(0,x)) \right| < \epsilon
	\end{split}
	\end{equation*} on $[0,t_{k_0}]$. The proof of \textbf{Claim} is complete. \end{proof}
  
\begin{remark} A few remarks are in order.
	\begin{itemize}
		\item \textit{Large Lagrangian deformation occurs at the origin}. Proposition \ref{prop:LLD-bubbles} shows that for bubbles satisfying $S_n \rightarrow \infty$ as $n \rightarrow \infty$, large Lagrangian deformation must occur, and it occurs even within a time interval that shrinks to zero for $n$ large. We emphasize that we can pinpoint the location of large Lagrangian deformation to be the origin (which was an open problem to the best of our knowledge), while using contradiction arguments it is possible (see \cite{BL,EJ}), with less work, to show existence of large Lagrangian deformation (somewhere in the domain). 

		\item \textit{Dichotomy for bubbles}. Note that in the case when the sequence $a_k$ is summable, the initial vorticity belongs to the critical Besov space $B^0_{\infty,1}$ uniformly in $n$ (for the rigorous calculation, see \cite{MY} for example). There is uniqueness and existence in this space $B^0_{\infty,1}$ (\cite{V}), which in particular guarantees that the corresponding velocity gradient is uniformly bounded in $n$ for a short time interval. Therefore, we have the following dichotomy for bubbles: short-time large Lagrangian deformation occurs \textit{if and only if} the sequence $\{ a_k \}$ is not summable. 
		
		\item \textit{Unbounded case}. Even when $\limsup_{k} a_k \rightarrow +\infty$ (i.e. when the sequence $\omega_{0,n}^L$ is not uniformly bounded in $L^\infty$), the lower bound \eqref{eq:LLD-bubbles} is still valid with a constant independent of $n$ but holds within a smaller time interval depending on $n$. One may follow the above proof except {that one needs to work with a more precise variant of \eqref{eq:keyLemma} and track the dependence of the error in $n$.}

		\item \textit{Sharpness of the growth rate}. It can be shown that with the data in \eqref{eq:bubbles}, we have \begin{equation*}
		\begin{split}
		|D\eta(t,0)| \le C(c_2),\qquad t \in [0,c_2/S_n]
		\end{split}
		\end{equation*} for any fixed constant $c_2 > 0$. This follows from the well-posedness of the Euler equations with vorticity in $B^0_{\infty,1}$ and the fact that $\nrm{\omega_{0,n}^L}_{B^0_{\infty,1}} \sim S_n$. Comparing this with \eqref{eq:LLD-bubbles}, one sees that the lower bound is sharp at least during this time scale. Hence we must wait a bit longer to see large deformation at the origin. 
		
		\item \textit{Large Lagrangian deformation in a small ball}. After appealing to an ODE argument almost identical to the one given in the proof of Proposition \ref{BC} using \eqref{eq:stretch-ODE}, it is possible to show that the growth rate stated in \eqref{eq:LLD-bubbles-shorttime} persists for points inside a small ball centered at the origin. The radius can be estimated similarly as well. 
		
		\item \textit{Hierarchical scale interaction}. In this setting of Bourgain and Li, we can see a hierarchical mechanism of multi-scale interaction 	of 2D-bubbles. We see that the behavior of large-scale bubble is not strongly affected by the smaller-scale bubbles {(as clearly demonstrated by the Key Lemma)}, however, not completely separative (c.f. the 3D-case \eqref{scale-separation-expression}). 
		
		\item \textit{Case of the continuum}. Our considerations equally apply well to the ``continuum'' version of the bubbles; that is, we may take locally \begin{equation*}
		\begin{split}
		\omega_{0,n}(r,\theta) = \varphi_{2^{-n-1}} * (g(r) \chi(\theta) ), \quad 0 \le r < 1/2
		\end{split}
		\end{equation*} where $\chi \ge 0$ on $\theta \in [0,\pi/2]$ and $\chi(\theta) = -\chi(-\theta) = -\chi(\pi-\theta)$, and $g \ge 0$ is a bounded continuous function on $[0,1/2] \rightarrow [0,1]$. Here $a_k$ corresponds to $g(2^{-k})$ and $S_k$ to $\int_{2^{-k}}^{1} g(r)r^{-1}dr$. For an example, in the case $g(r) = |\ln r|^{-1/2-\eps}$, $\omega_0 = g(r)\sin(2\theta)$ belongs to $H^1$ (considered explicitly in \cite{EJ}), and using the method in this paper one can show that the corresponding solution escapes $H^1$ without appealing to a contradiction argument. 
		
	\end{itemize}

\end{remark}

\section{Conclusion}

We prepared small-scale vortex blob and large-scale anti-parallel
vortex tubes for the initial data, and showed that the corresponding
3D Euler flow creates instantaneous vortex-stretching. {In turn, using this stretching, we showed that the corresponding 3D Navier-Stokes flow satisfies a modified version of zeroth-law}, which is the cornerstone of the
turbulence study field. {Thus this instantaneous vortex-stretching could be a key to make further progress in this field. We conjecture that an initial data which satisfies the actual 
zeroth-law should behave as in Goto-Saito-Kawahara's turbulence picture \cite{GSK}.}
In the Appendix, we mathematically formulated their turbulence picture and proved
Kolmogorov's $-5/3$-law, by assuming space-locality, scale-locality,
energy-flux in equilibrium state and space-filling (Frisch \cite{F}).

\appendix

{\section{Kolmogorov's $-5/3$-law}}

In this section we derive Kolmogorov's $-5/3$-law from GSK-type of self-similar hierarchy.
We first define energy spectrum based on  Constantin \cite{Constantin} (see also Mazzucato \cite{Mazzucato}). In order to do so, we recall the scale-locality on GSK's Navier-Stokes turbulence picture.
They say 2 to 8 times scale-interaction is dominant (scale-locality).
Thus we define $A^L_k$ as the large-scale spectrum and  $A^S_k$ as the  adjacent small-scale spectrum
involving this scale-locality property as follows: 
\begin{equation*}
A^L_k:=\{\xi\in\mathbb{R}^d:k/2\leq |\xi|\leq 2k\}
\quad\text{and}\quad A^S_k=bA^L_k
\end{equation*}
for some number $b\in \{2,3,\cdots,8\}$.

\begin{definition}[Energy spectrum based on Littlewood-Paley coarse-graining] 
Let us define Energy spectrum $E$ as follows:
\begin{equation*}
 E(t,k)=k^{-1}\int_{\mathbb{R}^d}\chi_{A^S_k\cup A^L_k}(\xi)|{\hat u}(t,\xi)|^2d\xi
\end{equation*}
where $\chi_{A^S_k\cup A^L_k}$ is the characteristic function on $A^S_k \cup A^L_k$.
\end{definition}
In order to propose a reasonable self-similar scaling assumption based of GSK's turbulence picture, we need to prepare several definitions.
Let $R_\theta$ be a 2D-rotation matrix such that
\begin{equation*}
R_\theta=
\begin{pmatrix}
\cos\theta& -\sin\theta & 0\\
\sin\theta& \cos\theta& 0\\
0& 0& 1
\end{pmatrix}.
\end{equation*}
We will use this $\theta$ as the random variable.
Let us recall
the second conclusion in GSK:
stretched vortex tubes tend to align in the direction perpendicular to larger-scale vortex tubes.
Ideally,  each hierarchy the velocities are orthogonal to the previous hierarchical level. 
However,  in the each hierarchy, randomness is already involved, so, in this paper, 
we do not take the orthogonality into account. 
Thus, we employ a set of rotation matrices $\{Q_j\}_j$ with random variables (random angles).
Let
\begin{equation*}
(R_\theta\circ \hat u)(\xi):=R_\theta\hat u(R^{-1}_\theta\xi)
\quad\text{and}\quad
(Q_j\circ \hat u)(\xi):=Q_j\hat u(Q_j^{-1}\xi).
\end{equation*}
Let $\hat U(\xi)$ ($\supp \hat U\subset  A^L_{1}\cup A^S_{1}$) be an ``snapshot of the minimum piece of ideal turbulence" (c.f. Subsection \ref{single bubble}),
and we assume the following symmetry:
\begin{equation*}
(R_{\pi/2}\circ\hat U)(\xi)=-\hat U(\xi).
\end{equation*}
For the corresponding probability distribution $\mu$, we assume a translational symmetry: 
\begin{equation*}
\mu([0,\theta))=\mu([\pi/2,\pi/2+\theta))\quad\text{for}\quad
\theta\in[0,2\pi).
\end{equation*}
These symmetries are corresponding to ``anti-parallel pairs of vortex tubes" in GSK's Navier-Stokes turbulence picture.
 Let us define Energy spectrum of $U$ as follows:
\begin{equation*}
 E^U=\int_{\mathbb{R}^d}\chi_{A^S_1\cup A^L_1}(\xi)|{\hat U}(\xi)|^2d\xi.
\end{equation*}
\begin{corollary}\label{mean-zero}
By the above symmetries, we see 
\begin{equation*}
\begin{split}
&\mathbb{E}[R_\theta\circ \hat U]
=\int_0^{2\pi}(R_\theta\circ \hat U) \mu(d\theta)\\
&=\int_0^{\pi/2}+\int_{\pi/2}^{\pi}+\int_{\pi}^{3\pi/2}+\int_{3\pi/2}^{2\pi}=
\int_0^{\pi/2}-\int_0^{\pi/2}+\int_0^{\pi/2}-\int_0^{\pi/2}=0.
\end{split}
\end{equation*}
Clearly $\mathbb{E}[|R_\theta\circ \hat U|^2]<\infty$.
\end{corollary}
Let $n$ be a natural number such that $k^d\sim b^{nd}$ ($d$ is dimension), so, $n\sim(\log k)/(\log b)$.
This $n$ exhibits the number of hierarchy.
Now we give a reasonable scaling assumption. Let
$\alpha\in\mathbb{R}$, $\{y(j)\}_{j}\subset \mathbb{R}^d$, $t>0$, $k>1$, 
rotation matrices $\{R_{\theta(j)}\}_{j}$, $\{Q_j\}_j$, assume $\hat u$ satisfies
\begin{equation}\label{scaling assumption}
\text{scaling assumption:}\quad \hat u(t,k\xi)=\sum_{j=1}^{b^{nd}}
 k^\alpha Q_{j}\circ R_{\theta(j)}\circ{\hat U}(\xi)e^{ik^{-1}y(j)\cdot \xi},
\end{equation}
for  $\xi\in  A^L_{1}\cup A^S_{1}$. We will determine the reasonable relation  between $k$ and $t$ later.
\begin{remark}
By the inverse Fourier transform, we have 
\begin{equation*}
k^{-d}u(t,k^{-1}x)=\sum_{j=1}^{b^{nd}}
k^{\alpha}(Q_j\circ R_{\theta(j)}\circ U)(x-k^{-1}y(j)).
\end{equation*}
\end{remark}

At this stage, $\alpha$ is unknown, so, first we need to figure out the exact value of $\alpha$.
Summing $b^{nd}$-elements, that is, $k^d$-elements comes from space-filling. See \cite[Section 7.3 and Figure 7.2]{F}. It says that the cascade according to
the Kolmogorov 1941 theory,
at each step the eddies are space-filling. 
For convenience of readers, we cite the most important paragraph from \cite{F}:
\vspace{0.5cm}
\begin{center}
\textit{Pheomenology of turbulence: The eddies of various sizes are represented as blobs stacked in
decreasing sizes.
The upper most eddies have a scale $\sim \ell_0$. The successive
generations of eddies
have scales $\ell_n\sim \ell_0r^n$ $(n=0,1,2\cdots)$, where $0<r<1$.
The value $r=1/2$ is the most common choice, but the exact value has no meaning.
The smallest eddies have scales $\sim\eta$, the Kolmogorov dissipation scale.
The number of eddies per unit volume is assumed to grow with $n$ as
$r^{-3n}$ to ensure that small eddies are as space-filling as large
ones.}
\end{center}
\vspace{0.5cm}

\begin{remark}
However from our vortex-stretching result  (Subsection \ref{single bubble}),
\begin{equation*}
|y(j)-y(j')|\lesssim 1\quad (j\not =j')
\end{equation*}
seems rather natural (not homogeneously distributing in the upper most eddies region). 
In this case summing  less than $b^{nd}$-elements
is rather suitable. This should be corresponding to the $\beta$-model (see \cite[Subsection 8.5]{F}).
Thus we need to suitably adjust various assumptions here to the $\beta$-model (intermittency), and this is our future work.
\end{remark}
\begin{definition}
Let us  define energy-flux as follows (this is based on \eqref{LP-version}):
\begin{equation*}
\Pi_k^{\Theta}(t):=
\int_{\mathbb{R}^d}\int_{\mathbb{R}^d}(\chi_{A^S_k}\hat u(t,\xi-\eta)\cdot i\xi)\chi_{A^S_k}\hat u(t,\eta)\chi_{A^L_k}\hat u(t,-\xi)d\eta d\xi,
\end{equation*}
where $\Theta:=(\{Q_j\}_j, \{\theta(j)\}_{j},\{y(j)\}_{j})$
are random variables, and assume that, at least, $\{\theta(j)\}_{j}$ are independent (this energy-flux is 
implicitly depending on $\Theta$).
Also let us define
\begin{equation*}
\Pi^{U}:=
\int_{\mathbb{R}^d}\int_{\mathbb{R}^d}(\chi_{A^S_1}\hat U(\xi-\eta)\cdot i\xi)\chi_{A^S_1}\hat U(\eta)\chi_{A^L_1}\hat U(-\xi)d\eta d\xi.
\end{equation*}
Similarly we define the three-wave-interaction $J$:
\begin{equation*}
J(\hat u,\hat v,\hat w):=
\int_{\mathbb{R}^d}\int_{\mathbb{R}^d}(\chi_{A^S_k}\hat u(t,\xi-\eta)\cdot i\xi)\chi_{A^S_k}\hat v(t,\eta)\chi_{A^L_k}\hat w(t,-\xi)d\eta d\xi.
\end{equation*}
\end{definition}
We now give an energy-flux assumption as follows:

\begin{equation}\label{energy-flux assumption}
\text{energy-flux assumption:}\quad 
\lim_{N\to\infty}\frac{\sum_{\ell=1}^N\Pi_{k}^{\Theta_\ell}(t)}{N}=\Pi^U=:\epsilon>0
\quad \text{for any}\quad t>0, 
\end{equation}
in stochastic convergence (c.f. \eqref{forward energy cascade}). Again, we determine the reasonable relation between $k$ and $t$ later.
This energy-flux assumption is rather natural in the turbulence study field:
``energy-flux" is independent of scale and equal to energy input/output, see \cite[Subsection 6.2.4]{F}.
\begin{remark}
However 
in GSK point of view, ``$t>0$" in \eqref{energy-flux assumption} is not accurate. Rigorously, we need to take ``quasi-periodicity" into account. See \cite[Section E]{GSK}.
We briefly mention the relation between $t$ and $n$ (the number of hierarchy) later.
\end{remark}
The energy spectrum $E$ is also implicitly depending on the  random variables $\Theta$.
We sometimes denote $E^{\Theta}(=E)$ to emphasize these dependences.
Then we now derive the Kolmogorov's $-5/3$-law.
\begin{theorem}
Assume the scaling-assumption and energy-flux-assumption. Then
we have 
\begin{equation*}
\frac{\sum_{\ell=1}^NE^{\Theta_\ell}(t,k)}{N}\to k^{-5/3}E^U\quad (N\to\infty)
\end{equation*}
in stochastic convergence.
\end{theorem}
\begin{remark}
By applying the usual H\"older and Young inequalities, \begin{equation*}
\epsilon  = \Pi^U=
J(\hat U,\hat U,\hat U)
\lesssim 
\|\chi_{A^L_1\cup A^S_1}\hat U\|_{L^2}^3
=(E^U)^{3/2}.
\end{equation*}
Thus there is a constant $C>0$ independent of $k$ (depending on 
the H\"older and Young inequalities themselves)
such that the following inequality holds:
\begin{equation*}
E(t,k)\geq C\epsilon^{2/3}k^{-5/3}.
\end{equation*}
To figure out the best upper bound, we may need to look into $U$ seriously. It should be interesting to figure out the relation between $U$ and the Kolmogorov constant.
\end{remark}

\begin{remark}
By dimensional analysis: $E\sim(length)^3/(time)^2$,  each stretching time can be estimated as  $\sim \epsilon^{-1/3}k^{-2/3}\sim \epsilon^{-1/3}b^{-2n/3} $ (see Goto \cite{G}).
Thus the time $t$ can be estimated as $t\sim \epsilon^{-1/3}\sum_{j=1}^{\log_b k}b^{-2/3j}\sim\sum_{j=1}^{n}b^{-2/3j}$.
This discrete summation must be related to quasi-periodicity \cite[Section E]{GSK}.
\end{remark}

\begin{proof}

Let $h_k\circ \hat U(\xi)=\hat u(t, k\xi)$.
A direct calculation yields
\begin{equation*}
\begin{split}
&\Pi^{\Theta}_k(t) =
J(\hat u(t),\hat u(t),\hat u(t))
=
k^{2d+1}J(h_k\circ\hat u(t),h_k\circ\hat u(t),h_k\circ\hat u(t))\\
&=
k^{2d+1+3\alpha}
J\bigg(\sum_{j_1=1}^{b^{nd}}
Q_{j_1}\circ R_{\theta(j_1)}\circ \hat Ue^{ik^{-1}y(j_1)\cdot \xi}, \sum_{j_2=1}^{b^{nd}}
Q_{j_2}\circ R_{\theta(j_2)}\circ \hat Ue^{ik^{-1}y(j_2)\cdot \xi},
\sum_{j_3=1}^{b^{nd}}
Q_{j_3}\circ R_{\theta(j_3)}\circ \hat Ue^{ik^{-1}y(j_3)\cdot \xi}\bigg).\\
\end{split}
\end{equation*}
By change of variables, we have 
\begin{equation*}
\begin{split}
&
J(Q_{j_1}\circ R_{\theta(j_1)}\circ \hat Ue^{ik^{-1}y(j_1)\cdot \xi},
Q_{j_2}\circ R_{\theta(j_2)}\circ \hat Ue^{ik^{-1}y(j_2)\cdot \xi}, 
Q_{j_3}\circ R_{\theta(j_3)}\circ \hat Ue^{ik^{-1}y(j_3)\cdot \xi})
=J(\hat U,\hat U,\hat U)
\end{split}
\end{equation*}
for $j_1=j_2=j_3$.
Thus
by  
decomposing into resonant and non-resonant parts, we have 
\begin{equation*}
\begin{split}
&\Pi^{\Theta}_k(t)  =
k^{2d+1+3\alpha+d}
J(\hat U,\hat U,\hat U) \\
& +  k^{2d+1+3\alpha}\sum_{\stackrel{j_1,j_2,j_3=1}{j_1\not=j_2\ \text{or}\ j_2\not=j_3}}^{b^{nd}}
J\bigg(Q_{j_1}\circ R_{\theta(j_1)}\circ \hat Ue^{ik^{-1}y(j_1)\cdot \xi},  Q_{j_2}\circ R_{\theta(j_2)}\circ \hat Ue^{ik^{-1}y(j_2)\cdot \xi},
Q_{j_3}\circ R_{\theta(j_3)}\circ \hat Ue^{ik^{-1}y(j_3)\cdot \xi}\bigg)\\
&\quad\quad  =:
k^{2d+1+3\alpha+d}
\Pi^U
+I^{\Theta}.
\end{split}
\end{equation*}
From Corollary \ref{mean-zero} and independence of random variables $\theta(j)$,
we see that $\mathbb{E}[I^{\Theta}]=0$.
Thus by Law of Large Numbers, we have 
\begin{eqnarray*}
\sum_{\ell=1}^N\frac{\Pi_k^{\Theta_\ell}(t)}{N}&\to&
k^{2d+1+3\alpha+d}
\Pi^U
\quad (N\to\infty)
\end{eqnarray*}
in stochastic convergence.
By \eqref{energy-flux assumption}, we have  $\alpha=-(1+3d)/3$.
Now we compute the energy spectrum. Let 
\begin{equation*}
L(\hat u,\hat v):=\int_{\mathbb{R}^d}\chi_{A^S_k\cup A^L_k}(\xi)\hat u(\xi)\cdot \bar {\hat v}(\xi)d\xi.
\end{equation*}
By a similar calculation as in the $\Pi^{\Theta}_k$ case, we have 
\begin{equation*}
\begin{split}
 &
E^{\Theta}(t,k)= k^{-1}L(\hat u(t),\hat u(t))
=
k^{-1+d}L(h_k\circ\hat u(t),h_k\circ\hat u(t))\\
&=
k^{-1+d}L\bigg(
\sum_{j_1=1}^{b^{nd}}
Q_{j_1}\circ R_{\theta(j_1)}\circ \hat Ue^{ik^{-1}y(j_1)\cdot \xi},
\sum_{j_2=1}^{b^{nd}}
Q_{j_2}\circ R_{\theta(j_2)}\circ \hat Ue^{ik^{-1}y(j_2)\cdot \xi}
\bigg).\\
\end{split}
\end{equation*}
Again, by change of variables, we have 
\begin{equation*}
L(Q_{j_1}\circ R_{\theta(j_1)}\circ \hat Ue^{ik^{-1}y(j_1)\cdot \xi},
Q_{j_2}\circ R_{\theta(j_2)}\circ \hat Ue^{ik^{-1}y(j_2)\cdot \xi})
=L(\hat U,\hat U)
\end{equation*}
for $j_1=j_2$.
Thus we have 
\begin{equation*}
\begin{split}
E^{\Theta}(t,k) &=
k^{-1+2d}L(k^{\alpha}
\hat U,k^{\alpha}
\hat U) +\sum_{\stackrel{j_1,j_2=1}{j_1\not=j_2}}^{b^{nd}}
L(Q_{j_1}\circ R_{\theta(j_1)}\circ \hat Ue^{ik^{-1}y(j_1)\cdot \xi},
Q_{j_2}\circ R_{\theta(j_2)}\circ \hat Ue^{ik^{-1}y(j_2)\cdot \xi}
)\\
&=:
k^{-1+2d+2\alpha}L(\hat U,\hat U)+\tilde I^{\Theta}
=k^{-5/3}E^U+\tilde I^{\Theta}.
\end{split}
\end{equation*}
Again, from Corollary \ref{mean-zero} and independence of random variables $\theta(j)$,
we see that $\mathbb{E}[\tilde I^{\Theta}]=0$. Therefore by Law of Large Numbers, we finally obtain
\begin{eqnarray*}
\sum_{\ell=1}^N\frac{E^{\Theta_\ell}(t,k)}{N}&\to&
k^{-5/3}
E^U\quad (N\to\infty)
\end{eqnarray*}
in stochastic convergence.
\end{proof}


\vspace{0.5cm}
\noindent
{\bf Acknowledgments.}\ 
{We thank Professors A. Mazzucato and T. D. Drivas for inspiring communications and telling us about the articles \cite{CD} and \cite{DE}, respectively. We are also grateful to Professors P. Constantin and T. Elgindi for valuable comments.}
Research of TY  was partially supported by 
Grant-in-Aid for Young Scientists A (17H04825),
Japan Society for the Promotion of Science (JSPS). IJ has been supported by the POSCO Science Fellowship of POSCO TJ Park Foundation. 
\bibliographystyle{amsplain} 

\begin{thebibliography}{10} 
\bibitem{BC}
H. Bahouri and J.-Y. Chemin,
\textit{\'{E}quations de transport relatives \'a\ des champs de vecteurs non-lipschitziens et m\'ecanique des fluides},
Arch. Rational Mech. Anal., \textbf{127} (1994), 159-181.



\bibitem{BKM}
J. T. Beale, T. Kato and A. Majda,
\textit{Remarks on the breakdown of smooth solutions for the 3-D Euler equations},
Commun. Math. Phys., \textbf{94} (1984), 61-66.



\bibitem{BL} 
J. Bourgain and D. Li, 
\textit{Strong ill-posedness of the incompressible Euler equations in borderline Sobolev spaces}, 
Invent. math. \textbf{201} (2015), 97-157; 
preprint arXiv:1307.7090 [math.AP]. 


\bibitem{BLSV}
 Buckmaster T., De Lellis C., Sz\'ekelyhidi Jr L., Vicol V.,
\textit{Onsager's conjecture for admissible weak solutions}, 
arXiv:1701.08678.









\bibitem{Constantin}
P. Constantin,
\textit{The Littlewood-Paley Spectrum in Two-Dimensional Turbulence}, 
Theo.  Comput. Fluid Dynam., \textit{9}, (1997), 183-189. 









\bibitem{CD}
G. Crippa and C. De Lellis,
\textit{Estimates and regularity results for the DiPerna-Lions flow},
J. reine angew. Math., \textit{616}, (2008), 15-46. 




\bibitem{Dr}
T. D. Drivas, 
\textit{Turbulent cascade direction and Lagrangian time-asymmetry},
J. Nonlinear Sci., (2018), 1-24.

\bibitem{DE}
T. D. Drivas and G. L. Eyink, 
\textit{An Onsager Singularity Theorem for Leray Solutions of Incompressible Navier-Stokes},
 arXiv:1710.05205.


\bibitem{EJ}
T. Elgindi and I.-J. Jeong, 
\textit{Ill-posedness for the incompressible Euler equations in critical Sobolev spaces},
Ann. PDE, \textbf{3} (2017), 7.

\bibitem{E06-0} 
G. L. Eyink, 
\textit{Multi-scale gradient expansion of the turbulent stress tensor},  
J. Fluid Mech., \textbf{549} (2006), 159-190.



\bibitem{E3}
G. L. Eyink,
\textit{Review of the Onsager ``Ideal Turbulence" Theory},
arXiv:1803.02223. 

\bibitem{EM} 
T. Elgindi and N. Masmoudi, 
\textit{$L^\infty$ ill-posedness for a class of equations arising in hydrodynamics}, 
preprint arXiv:1405.2478 [math.AP]. 




\bibitem{F} 
U. Frisch, 
\textit{Turbulence}, 
Cambridge University Press, Cambridge 1995. 



\bibitem{G}
S. Goto,
\textit{
Developed Turbulence: On the Energy Cascade},
The Nihon Butsuri Gakkaishi (Butsuri), \textbf{73} (2018) 457-462. 

\bibitem{GSK}
S. Goto, Y. Saito, and G. Kawahara,
\textit{Hierarchy of antiparallel vortex tubes in spatially periodic turbulence at high Reynolds numbers},
Phys. Rev. Fluids \textbf{2} (2017), 064603. 







\bibitem{KS}
A. Kiselev and  V. Sverak
\emph{Small scale creation for solutions of the incompressible two-dimensional Euler equation.}      
Annals of Math., 180 (2014), 1205-1220. 








\bibitem{LS1}
X. Luo and R. Shvydkoy, 
\textit{2D Homogeneous Solutions to the Euler Equation}, 
Comm. Partial  Diff. Eq., \textbf{40:9}  (2015),  1666-1687.

\bibitem{LS2}
X. Luo and R. Shvydkoy,
\textit{Addendum: 2D homogeneous solutions to the Euler equation}, 
Comm. Partial Diff. Eq., \textbf{42:3}  (2017),  491-493.



\bibitem{Mazzucato}
A. L. Mazzucato, \textit{On the energy spectrum for weak solutions of the Navier-Stokes equations},
 Nonlinearity  \textit{18}  (2005), 1-19.

\bibitem{MK1}
H. Miura and S. Kida,
\textit{Identification of tubular vortices in turbulence},
J. Phys. Soc. Japan, \textbf{66} (1997), 1331.



\bibitem{MK2}
H. Miura and S. Kida,
\textit{Swirl condition in low-pressure vortex},
J. Phys. Soc. Japan, \textbf{67} (1998), 2166.




\bibitem{MY} 
G. Misio{\l}ek and T. Yoneda, 
\textit{Continuity of the solution map of the Euler equations in H\"older spaces 
and weak norm inflation in Besov spaces}, 
Trans. Amer. Math. Soc., \textbf{370} (2018), 4709-4730. 




\bibitem{Sh}
R. Shvydkoy,
\textit{Homogeneous solutions to the 3D Euler system},
Trans. Amer. Math. Soc., \textbf{370} (2018), 2517-2535.







\bibitem{V}
M. Vishik,
\textit{Hydrodynamics in Besov spaces}, 
Arch. Ration. Mech. Anal., \textbf{145} (1998), 197-214.





\bibitem{Z} 
A. Zlatos, 
\emph{Exponential growth of the vorticity gradient for the Euler equation on the torus}, 
 Adv. Math., \textbf{268} (2015), 396-403.

\end{thebibliography}

\end{document}